\documentclass[12pt]{article}

\usepackage{amsmath,amsthm}
\usepackage{amsfonts}
\usepackage{amssymb}
\usepackage{color}
\usepackage{mathrsfs}
\usepackage{graphicx}
\usepackage{caption}

\newtheorem{thm}{Theorem}[section]
\newtheorem{defn}[thm]{Definition}

\newtheorem{lem}[thm]{Lemma}
\newtheorem{cor}[thm]{Corollary}
\newtheorem{rmk}[thm]{Remark}

\def\D{\mathrm{D}}

\def\M{\mathscr{M}}

\def\O{\mathscr{O}}

\def\R{\mathscr{R}}
\def\S{\mathscr{S}}
\def\Z{\mathscr{Z}}

\def\c{\mathrm{c}}
\def\d{\mathrm{d}}
\def\e{\mathrm{e}}
\def\h{\mathrm{h}}
\def\s{\mathrm{s}}
\def\u{\mathrm{u}}

\def\Cset{\mathbb{C}}
\def\Kset{\mathbb{K}}
\def\Lset{\mathbb{L}}
\def\Nset{\mathbb{N}}
\def\Rset{\mathbb{R}}
\def\Zset{\mathbb{Z}}

\def\id{\mathrm{id}}
\def\loc{\mathrm{loc}}
\def\Re{\mathrm{Re}}
\def\Span{\mathrm{span}}

\def\sech{\mathrm{sech}}

\def\half{{\textstyle\frac{1}{2}}}
\def\third{{\textstyle\frac{1}{3}}}
\def\fourth{{\textstyle\frac{1}{4}}}
\def\sixth{{\textstyle\frac{1}{6}}}

\def\epsilon{\varepsilon}

\def\eg{{e.g.}}
\def\etal{\textit{et al.}}
\def\ie{{i.e.}}

\renewcommand{\topfraction}1
\renewcommand{\bottomfraction}1

\begin{document}


\title{Analytic and algebraic conditions for bifurcations of homoclinic orbits I:
Saddle equilibria}

\author{David~Bl\'azquez-Sanz \& Kazuyuki~Yagasaki}
\date{May 23, 2010}

\maketitle

\begin{abstract}
We study bifurcations of homoclinic orbits to hyperbolic saddle equilibria
in a class of four-dimensional systems which may be Hamiltonian or not.
Only one parameter is enough
to treat these types of bifurcations in Hamiltonian systems
but two parameters are needed in general systems.
We apply a version of Melnikov's method due to Gruendler
to obtain saddle-node and pitchfork types of bifurcation results
for homoclinic orbits.
Furthermore we prove that if these bifurcations occur,
then the variational equations around the homoclinic orbits
are integrable in the meaning of differential Galois theory
under the assumption
that the homoclinic orbits lie on analytic invariant manifolds.
We illustrate our theories with an example
which arises as stationary states of
coupled real Ginzburg-Landau partial differential equations, 
and demonstrate the theoretical results by numerical ones.\\[1ex]
\end{abstract}

\section{Introduction}

Bifurcations of homoclinic orbits in ordinary differential equations
 have been studied in numerous articles over the past decades.
They also arise
 as bifurcations of solitons or pulses in partial differential equations (PDEs),
 and have attracted much attention even in the fields of PDEs and nonlinear waves
 (see, \eg, Section~2 of \cite{S02}).

Many of researches on this topic are related to bifurcations
 from single-bump homoclinic orbits to multi-bump ones.
See Section~5.2.4 of \cite{S02}
 for a concise and useful review of the references.
As an exceptional study, Knobloch \cite{K97} showed that
 a saddle-node type of bifurcation
 for homoclinic orbits to hyperbolic saddles can occur
 in reversible and conservative systems with one parameter
 when the homoclinic orbits are non-transversal,
 \ie, the tangent spaces to the stable and unstable manifolds of the saddles
 along them have two-dimensional intersections,
 using so-called Lin's method \cite{L90}.
We can also use a version of Melnikov's method \cite{GH83,M63}
 due to Gruendler \cite{G92},
 who studied the persistence of homoclinic orbits,
 to treat such bifurcations
 in general higher-dimensional systems with two parameters.
We note that only one parameter is enough
 to treat these bifurcations
 for reversible and conservative systems
 but two parameters are needed in general systems.
A pitchfork type of bifurcation
 for homoclinic orbits to saddle-centers
 was also detected in reversible systems \cite{YW06},
 using an idea similar to that of Melnikov's method.

\begin{figure}
\begin{center}
\includegraphics[scale=0.8]{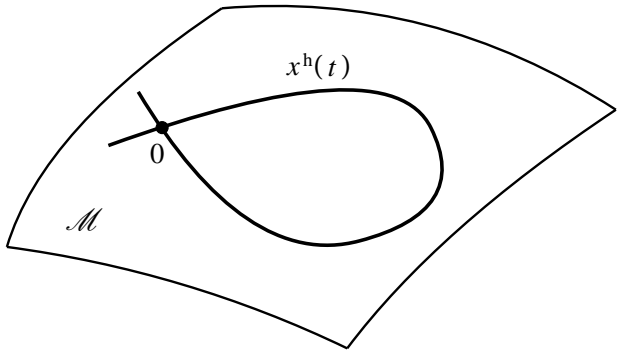}
\caption{Assumption~(A2)}
\label{fig:A}
\end{center}
\end{figure}

Here we are interested in the latter,
 saddle-node and pitchfork types of bifurcations
 for homoclinic orbits to hyperbolic saddles in systems of the form
\begin{equation}
\dot{x}=f(x;\mu),\quad
x\in\Rset^n,\quad
\mu\in\Rset^m,
\label{eqn:sys}
\end{equation}
where $f:\Rset^n\times\Rset^m\rightarrow\Rset^n$ is analytic,
 $\mu$ is a parameter (vector), $n$ is a positive integer and $m=1$ or 2.
Note that Eq.~\eqref{eqn:sys} may be Hamiltonian or not.
We consider the case of $n=4$ and make the following assumptions.
\begin{enumerate}
\item[\bf(A1)]
The origin $x=0$ is a hyperbolic saddle equilibrium in \eqref{eqn:sys} at $\mu=0$,
 such that $\D_x f(0;0)$ has four real eigenvalues,
 $\lambda_1\le\lambda_2<0<\lambda_3\le\lambda_4$.
\item[\bf(A2)]
At $\mu=0$ the hyperbolic saddle $x=0$ has a homoclinic orbit $x^\h(t)$.
Moreover, there exists a two-dimensional analytic invariant manifold $\M$
 containing $x=0$ and $x^\h(t)$ (see Fig.~\ref{fig:A}).
\end{enumerate} 
From (A1) we can also assume that
 the origin $x=0$ is still a hyperbolic saddle near $\mu=0$
 under some change of coordinates.
Assumption (A2) implies that
 the complexification of the vector field \eqref{eqn:sys} on $\M$ with $\mu=0$
 has a \emph{complex} separatrix $\Gamma$
 whose real part is the \emph{real} homoclinic orbit $x^\h(t)$.
In particular, the existence of the invariant manifold $\M$ is essential
 in Theorem~\ref{thm:main} below.
We also remark that as well known \cite{H96,S00,S98},
 there exists a center manifold for $x^\h(t)$
 under a more general condition,
 but it is not analytic in general.

As seen in \cite{G92,K97,S02},
 \emph{variational equations (VEs)} play an important role
 in persistence and bifurcation problems of homoclinic orbits.
The VE of \eqref{eqn:sys}
 around the homoclinic orbit at $\mu=0$ is given by
\begin{equation}
\dot{\xi}=\D_x f(x^\h(t);0)\xi,
\label{eqn:ve}
\end{equation}
where $\xi\in\Rset^n$.
We easily see that
 $\xi=\dot{x}^\h(t)$ is a bounded solution of \eqref{eqn:ve}
 tending to zero exponentially as $t\rightarrow\pm\infty$.
We consider the following condition.
\begin{enumerate}
\item
The VE~\eqref{eqn:ve} has another independent bounded solution.
\end{enumerate}
In Section~2,
 using Gruendler's version of Melnikov's method \cite{G92}
 instead of Lin's method \cite{L90},
 we prove the following results for \eqref{eqn:sys} with $n\ge 4$:
\begin{enumerate}
\item[(i)]
If condition~(C) holds,
 then a saddle-node or pitchfork bifurcation of homoclinic orbits occurs
 under some nondegenerate condition;
\item[(ii)]
If condition~(C) does not hold,
 then these bifurcations do not occur.
\end{enumerate}
Thus, condition~(C) provides a criterion
 for saddle-node and pitchfork bifurcations of homoclinic orbits
 in \eqref{eqn:sys}.

On the other hand,
 the problem of integrability is especially important
 in dynamical systems, along with bifurcations and chaos.
In particular,
 it has been extensively studied for two classes of dynamical systems,
 linear differential equations and Hamiltonian systems.
See, \eg, \cite{A89,VS03} for the notions of integrability
 in these two classes.
The former is also briefly described in Section~3.
Differential Galois theory is one for the former class
 and provides algorithms to solve integrable linear differential equations.
Morales-Ruiz and Ramis \cite{M99,MR01}
 applied the differential Galois theory
 and uncovered a relation between integrability of both classes:
If a Hamiltonian system is integrable near its particular solution
 in the context of complex Hamiltonian systems
 (\ie, in the sense of Liouville with meromorphic first integrals),
 then the related VE is integrable
 in the context of linear differential equations
 (\ie, in the meaning of differential Galois theory).
Their result was further extended to higher-order VEs \cite{MRS07}.
An implication of their result for the occurrence of chaotic dynamics
 in two-degree-of-freedom Hamiltonian systems with saddle-centers
 was also discussed in \cite{MP99,Y03}.

In this series of articles,
 we give another application of the differential Galois theory
 and clarify a connection of saddle-node and pitchfork bifurcations
 of homoclinic orbits in \eqref{eqn:sys}
 with the integrability of the VE~\eqref{eqn:ve}
 in the meaning of differential Galois theory.
More precisely, we prove the following theorem
(The proof is given in Section~4).
\begin{thm}
\label{thm:main}
Suppose that condition~(C) holds.
Then the VE~\eqref{eqn:ve}
 has a triangularizable differential Galois group,
when regarded as a complex differential equation
 with meromorphic coefficients
 in a desingularized neighborhood $\Gamma_\loc$ of the homoclinic orbit $x^\h(t)$
 in $\Cset^4$.
\end{thm}
This theorem means that under condition~(C)
 the VE~\eqref{eqn:ve} is integrable
 in the meaning of differential Galois theory.
See Sections~3 and 4 for more details on the statement.
Such a relation between integrability in the differential Galois setting
and the existence of bounded solutions for linear differential equations was empirically noticed in the context of exactly solvable potentials
 in one dimensional Schr\"odinger equations
 by Acosta-Humanez \etal\ \cite{AMW10}.
Here we concentrate ourselves on the case of general or Hamiltonian systems
 with hyperbolic saddle equilibria.
The case of reversible systems or saddle-center equilibria
 will be treated in part~II.

The outline of this paper is as follows:
In Section~2 we extend Gruendler's version of Melnikov's method \cite{G92}
 and present bifurcation results
 for homoclinic orbits not only in Hamiltonian systems
 but also in more general systems.
These results are easily obtained from the result of \cite{G92}
 but many of them are not found in other references, to the authors' knowledge.
We give necessary information on differential Galois theory in Section~3
 and state a proof of Theorem~\ref{thm:main} in Section~4.
In the proof, under assumption~(A2),
 we reduce integrability of the VE~\eqref{eqn:ve}
 to that of a two-dimensional linear system
 which consists of components normal to the invariant manifold $\M$
 and is called the \emph{normal variational equation (NVE)}.
Finally, in Section~5, we illustrate our theories with an example
 which arises as stationary states of 
 coupled real Ginzburg-Landau PDEs.
The theoretical results are also demonstrated by numerical ones
 using a numerical computation tool called AUTO97 \cite{AUTO97}.

\section{Melnikov analyses}

In this section,
 we extend Gruendler's version of Melnikov's method \cite{G92}
 and discuss the persistence and bifurcations of homoclinic orbits
 in \eqref{eqn:sys}.

\subsection{Extension of Melnikov's method}
We consider a more general situation of $n\ge 2$ in \eqref{eqn:sys}
 and assume the followings:
\begin{enumerate}
\item[\bf(M1)]
The origin $x=0$ is a hyperbolic saddle in \eqref{eqn:sys} at $\mu=0$
 such that $\D_x f(0;0)$ has $n_\s$ and $n_\u$ eigenvalues
 with negative and positive real parts, respectively, where $n_\s+n_\u=n$.
\item[\bf(M2)]
The equilibrium $x=0$ has a homoclinic orbit $x^\h(t)$ at $\mu=0$.
\item[\bf(M3)]
The VE~(\ref{eqn:ve}) has just $n_0$ independent bounded solutions,
 $\xi=\varphi_1(t)=x^\h(t),\varphi_2(t),\ldots,\varphi_{n_0}(t)$.
\end{enumerate}
Assumptions~(M1) and (M2) mean that
 the stable and unstable manifolds of the origin at $\mu=0$,
 $W_0^\s(0)$ and $W_0^\u(0)$, are of dimension $n_\s$ and $n_\u$, respectively,
 and follow from (A1) and (A2) for $n=4$ with $n_\s=n_\u=2$.
By assuming (M3), we have $n_\s,n_u\ge n_0$ and
\[
\dim(T_{x}W_0^\s(0)\cap T_{x}W_0^\u(0))=n_0
\]
along the homoclinic orbit $x^\h(t)$.
Note that
 the existence of an analytic invariant manifold
 containing the equilibrium $x=0$ and homoclinic orbit $x^\h(t)$
 is not required.

Using Theorem~1 of \cite{G92}, we obtain the following lemma immediately.
\begin{lem}
\label{lem:m1}
There exists a fundamental matrix
 $\Phi(t)=(\varphi_1(t),\ldots,\varphi_n(t))$ of (\ref{eqn:ve}) such that
 for nonnegative integers $k_j$, $j=1,\ldots,n$,
 and a permutation $\sigma$ on $n$ symbols $\{1,\ldots,n\}$,
\begin{align*}
&
\varphi_j(t)t^{-k_j}\e^{-\Re(\tilde{\lambda}_j)t}=\O(1)
 \quad\mbox{as $t\rightarrow+\infty$},\\
&
\varphi_j(t)t^{-k_{\sigma(j)}}\e^{-\Re(\tilde{\lambda}_{\sigma(j)})t}=\O(1)
 \quad\mbox{as $t\rightarrow-\infty$},
\end{align*}
where $\tilde{\lambda}_j$, $j=1,\ldots,n$, denote the eigenvalues of $\D_x f(0;0)$
 repeated according to algebraic multiplicity,
 and $\Re(\tilde{\lambda}_j)$ is negative if $j\le n_\s$ and positive if $j>n_\u$.
\end{lem}

By assumption~(M3) and Lemma~\ref{lem:m1} we have
\begin{equation}
\begin{array}{ll}
\displaystyle
 \lim_{t\rightarrow\pm\infty}\varphi_j(t)=0 & \mbox{for $j=1,\ldots,n_0$};\\
\displaystyle
 \lim_{t\rightarrow+\infty}\varphi_j(t)=0,\quad
 \lim_{t\rightarrow-\infty}\varphi_j(t)=\infty &\mbox{for $j=n_0+1,\ldots,n_s$},
\end{array}
\label{eqn:vphi1}
\end{equation}
and assume that
\begin{equation}
\begin{array}{ll}
\displaystyle
 \lim_{t\rightarrow\pm\infty}\varphi_j(t)=\infty
 &\mbox{for $j=n_\s+1,\ldots,n_\s+n_0$};\\
\displaystyle
 \lim_{t\rightarrow \infty}\varphi_j(t)=\infty,\quad
 \lim_{t\rightarrow -\infty}\varphi_j(t)=0 & \mbox{for $j=n_s+n_0+1,\ldots,n$}.
\end{array}
\label{eqn:vphi2}
\end{equation}
Define $\psi_j(t)$ for each $j=1,\ldots,n$ by
\[
\langle\psi_j(t),\varphi_k(t)\rangle=\delta_{jk},
\]
where $\delta_{jk}$ is Kronecker's delta,
 and $\langle\xi,\eta\rangle$ represents the inner product of $\xi,\eta\in\Rset^n$.
It immediately follows from \eqref{eqn:vphi1} and \eqref{eqn:vphi2} that
\begin{equation}
\begin{array}{ll}
\displaystyle
 \lim_{t\rightarrow\pm\infty}\psi_j(t)=\infty & \mbox{for $j=1,\ldots,n_0$};\\
\displaystyle
 \lim_{t\rightarrow+\infty}\psi_j(t)=\infty,\quad
 \lim_{t\rightarrow-\infty}\psi_j(t)=0 &\mbox{for $j=n_0+1,\ldots,n_s$};\\
\displaystyle
 \lim_{t\rightarrow\pm\infty}\psi_j(t)=0
 &\mbox{for $j=n_\s+1,\ldots,n_\s+n_0$};\\
\displaystyle
 \lim_{t\rightarrow \infty}\psi_j(t)=0,\quad
 \lim_{t\rightarrow -\infty}\psi_j(t)=\infty & \mbox{for $j=n_s+n_0+1,\ldots,n$}.
\end{array}
\label{eqn:psi}
\end{equation}
The functions $\psi_j(t)$, $j=1,\ldots,n$, can be obtained
 by the formula $\Psi(t)=(\Phi^*(t))^{-1}$,
 where $\Psi(t)=(\psi_1(t),\ldots,\psi_n(t))$
 and $\Phi^*(t)$ is the transpose matrix of $\Phi(t)$.
Moreover, $\Psi(t)$ is a fundamental matrix to the adjoint equation
\begin{equation}
\dot{\xi}=-\D_x f(x^\h(t);0)^* \xi
\label{eqn:ave}
\end{equation}
since $\Psi^*(t)\Phi(t)=\id_n$ with $\id_n$ the $n\times n$ identity matrix,
 so that $\Psi^{\prime *}(t)\Phi(t)+\Psi^*(t)\Phi'(t)=0$, \ie, 
\[
\dot{\Psi}(t)=-(\Psi^*(t)\dot{\Phi}(t)\Phi^{-1}(t))^*
 =-\D_x f(x^\h(t);0)^*\Psi(t).
\]

Now we look for a homoclinic orbit of the form
\begin{equation}
x=x^\h(t)+\sum_{j=1}^{n_0-1}\alpha_j\varphi_{j+1}(t)+\O(\sqrt{|\alpha|^4+|\mu|^2})
\label{eqn:ho}
\end{equation}
in (\ref{eqn:sys}) with $\mu\neq 0$, where $\alpha=(\alpha_1,\ldots,\alpha_{n_0-1})$.
Here the $\O(\alpha)$-terms are ignored in \eqref{eqn:ho} if $n_0=1$.

Let $\kappa$ be a positive real number such that
\[
\kappa<\fourth\,|\Re(\tilde{\lambda}_j)|,\quad
j=1,\ldots,n,
\]
and define two Banach spaces
\begin{eqnarray*}
\Z^0 &=& \{z\in C^0(\Rset,\Rset^n)\,|\,\sup_t|z(t)|e^{\kappa|t|}<\infty\},\\
\Z^1 &=& \{z\in C^1(\Rset,\Rset^n)\,|\,\sup_t|z(t)|e^{\kappa|t|}<\infty,
 \sup_x|\dot{z}(t)|e^{\kappa|t|}<\infty\},
\end{eqnarray*}
where the maximum of the suprema is taken as a norm of each spaces.
The following lemma is frequently used in the rest of this section.
\begin{lem}
\label{lem:m2}
Let $z\in\Z^1$.
Then we have
\[
\int_{-\infty}^\infty\langle\psi_{n_\s+j}(t),
\dot{z}(t)-\D_x f(x^\h(t);0)z(t)\rangle\,\d t=0,\quad
j=1,\ldots,n_0.
\]
\end{lem}
\begin{proof}
The proof is found dividedly in \cite{G92}.
Here it is given for the reader's convenience.

We easily see that for $z\in\Z^0$
\[
|\langle\psi_{n_\s+j}(t),z(t)\rangle|\le K_0 \|z\|\e^{-\kappa t},\quad
j=1,\ldots,n_0,
\]
by Lemma~\ref{lem:m1} and \eqref{eqn:psi},
where $K_0$ is some constant independent of $z$.
On the other hand,
\begin{align*}
&
\langle\psi_{n_\s+j}(t),\dot{z}(t)-\D_x f(x^\h(t);0)z(t)\rangle\\
&
=\langle\psi_{n_\s+j}(t),\dot{z}(t)\rangle
 -\langle\D_x f(x^\h(t);0)^*\psi_{n_\s+j}(t),z(t)\rangle\\
&
=\langle\psi_{n_\s+j}(t),\dot{z}(t)\rangle
 +\langle\dot{\psi}_{n_\s+j}(t),z(t)\rangle
=\frac{\d}{\d t}\langle\psi_{n_\s+j}(t),z(t)\rangle
\end{align*}
since $\psi_{n_\s+j}(t)$ is a solution \eqref{eqn:ave}.
Thus, we obtain the result.
\end{proof}
From Theorem~4 of \cite{G92} we also have the following lemma.
\begin{lem}
\label{lem:m3}
The nonhomogeneous VE,
\begin{equation}
\dot{\xi}=\D_x f(x^\h(t);0)\xi+\eta(t)
\label{eqn:nhve}
\end{equation}
with $\eta\in\Z^0$, has a solution in $\Z^1$ if and only if
\begin{equation}
\int_{-\infty}^{\infty}
 \langle\psi_{n_\s+j}(t),\eta(t)\rangle\,dt=0,\quad
j=1,\ldots,n_0.
\label{eqn:lem3a}
\end{equation}
Moreover, if condition (\ref{eqn:lem3a}) holds,
 then there exists a unique solution to (\ref{eqn:nhve})
 satisfying $\langle\psi_j(0),\xi(0)\rangle=0$ for $j=1,\ldots,n_0$.
\end{lem}

Let
\[
\bar{\Z}^1=\{z\in\Z^1\,|\,\langle\psi_j(0),z(0)\rangle=0,j=1,\ldots,n_0\}
 \subset\Z^1
\]
As the kernel of a continuous linear map
 $\bar{\Z}^1$ is closed and hence it is also a Banach space.
We define a differentiable function
 $F:\bar{\Z}^1\times\Rset^{n_0-1}\times\Rset^m\rightarrow\Z^0$ as
\begin{align}
F(z;\alpha,\mu)=&
\dot{x}^\h(t)+\dot{z}(t)
 +\sum_{j=1}^{n_0-1}\alpha_j\dot{\varphi}_{j+1}(t)\notag\\
& -f\left(x^\h(t)+z(t)+\sum_{j=1}^{n_0-1}\alpha_j\varphi_{j+1}(t);\mu\right).
\label{eqn:F}
\end{align}
A solution $z\in\bar{\Z}^1$ to
\begin{equation}
F(z;\alpha,\mu)=0
\label{eqn:F0}
\end{equation}
for $(\alpha,\mu)$ fixed gives a homoclinic orbit to the origin.

Let $q:\Rset\rightarrow\Rset$ be a continuous function satisfying
\[
\sup_t|q(t)|e^{\kappa t}<\infty
\quad\mbox{and}\quad
\int_{-\infty}^\infty q(t)\d t=1.
\]
Define a projection $\Pi:\Z^0\rightarrow\Z^0$ by
\[
\Pi(z)(t)
=q(t)\sum_{j=1}^{n_0}\left(\int_{-\infty}^\infty
 \langle\psi_{n_\s+j}(\tau),z(\tau)\rangle\,d\tau\right)\varphi_{n_\s+j}(t).
\]
Condition (\ref{eqn:F0}) is divided into two parts:
\begin{equation}
(\id-\Pi)F(z;\alpha,\mu)=0
\label{eqn:F1}
\end{equation}
and
\begin{equation}
\Pi F(z;\alpha,\mu)=0,
\label{eqn:F2}
\end{equation}
where ``$\id$'' represents the identity.
Obviously,
\begin{equation}
(\id-\Pi)F(0;0,0)=0.
\label{eqn:F10}
\end{equation}
We easily see that
\begin{equation}
\int_{-\infty}^\infty
 \langle\psi_{n_\s+j}(t),(\id-\Pi)z(t)\rangle\,\d t=0,\quad
j=1,\ldots,n_0,
\label{eqn:id-Pi}
\end{equation}
for $z\in\Z^0$.
Hence, by Lemma~\ref{lem:m3},
 Eq.~(\ref{eqn:nhve}) has a unique solution in $\bar{\Z}_1$
 if $\eta$ is involved in the range of $(\id-\Pi)$, $\R(\id-\Pi)$.
This means that for $\eta\in\R(\id-\Pi)$
 there is a unique function $\xi(t)\in\bar{\Z}_1$ such that
\[
\D_z(\id-\Pi)F(0;0,0)\xi=\eta.
\]
Using this fact and (\ref{eqn:F10}),
 we apply the implicit function theorem (\eg, Theorem~2.3 of \cite{CH82})
 to show that there are a neighborhood $U$ of $(\alpha,\mu)=(0,0)$
 and a differentiable function $\bar{z}:U\rightarrow\bar{\Z}^1$
 such that $\bar{z}(0,0)=0$ and 
\begin{equation}
(\id-\Pi)F(\bar{z}(\alpha,\mu);\alpha,\mu)=0
\label{eqn:F1a}
\end{equation}
for $(\alpha,\mu)\in U$.

Let
\begin{equation}
\bar{F}_j(\alpha,\mu)=\int_{-\infty}^\infty
 \langle\psi_{n_\s+j}(t),F(\bar{z}(\alpha,\mu)(t);\alpha,\mu)\rangle\,\d t,\quad
j=1,\ldots,n_0.
\label{eqn:bF}
\end{equation}
If $\bar{F}(\alpha,\mu)
 \equiv(\bar{F}_1(\alpha,\mu),\ldots,\bar{F}_{n_0}(\alpha,\mu))=0$,
 then $z=\bar{z}(\alpha,\mu)$ satisfies (\ref{eqn:F1}) and (\ref{eqn:F2}).
Thus, we can prove the following theorem as in Theorem~5 of \cite{G92}.

\begin{thm}
\label{thm:m}
In a neighborhood $U$ of $(\alpha,\mu)=(0,0)$,
 Eq.~\eqref{eqn:sys} has a unique homoclinic orbit of the form \eqref{eqn:ho}
 to the origin $x=0$ if $\bar{F}(\alpha,\mu)=0$.
\end{thm}

Henceforth we apply Theorem~\ref{thm:m}
 to prove persistence and bifurcation theorems for homoclinic orbits
 in \eqref{eqn:sys} with $n\ge 2$.
The first two results (Theorems~\ref{thm:m1} and \ref{thm:m2})
 are essentially the the same as Corollaries~7 and 8 of \cite{G92}
 but they are included with proofs
 for our unified presentation and the reader's convenience.
We will also need information on higher-order derivatives of $\bar{F}$
 up to third-order at $(\alpha,\mu)=(0,0)$,
 some of which were not required in \cite{G92}.
 
\subsection{General case}

\begin{figure}[t]
\begin{center}
\includegraphics[scale=0.7]{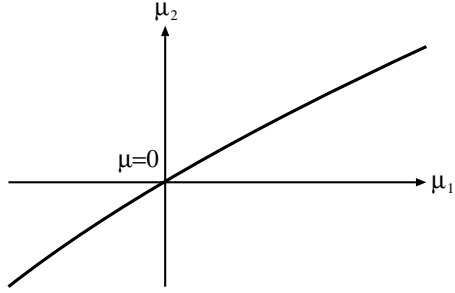}
\caption{Persistence of homoclinic orbits}
\label{fig:m1}
\end{center}
\end{figure}

Let $m=2$ and assume that $n_0=1$,
 which means that condition~(C) does not hold when $n=4$.
Then $\bar{F}$ is independent of $\alpha$.
Define two constants $a_{11},a_{12}\in\Rset$ as
\begin{equation}
a_{1j}=\int_{-\infty}^{\infty}
 \langle\psi_{n_\s+1}(t),\D_{\mu_j}f(x^\h(t);0)\rangle\,\d t.
\label{eqn:a}
\end{equation}

\begin{thm}
\label{thm:m1}
Under assumptions (M1)-(M3)
 let $n_0=1$ and suppose that $a_1=(a_{11},a_{12})\neq(0,0)$.
Then for some open interval $I$ including $0$,
 there exists a differentiable function
 $\phi_j:I\rightarrow\Rset$ with $\phi_j(0)=0$, $j=1$ or 2,
 such that a homoclinic orbit exists
 for $\mu_1=\phi_1(\mu_2)$ or $\mu_2=\phi_2(\mu_1)$ (see Fig.~\ref{fig:m1}).
\end{thm}

\begin{proof}
We simply write $\bar{z}=\bar{z}(\mu)$.
For $\mu\in\Rset^2$ we compute \eqref{eqn:bF} as
\begin{align*}
&
\bar{F}_1(\mu)\\
&
=\int_{-\infty}^\infty\langle\psi_{n_\s+1}(t),
 \dot{\bar{z}}(t)-\D_x f(x^\h(t);0){\bar{z}}(t)-\D_\mu f(x^\h(t);0)\mu\rangle\d t
 +\O(|\mu|^2)\\
&
=-a_1\mu+\O(|\mu|^2).
\end{align*}
Here we used the fact that $\bar{z}(0,0)=0$ and Lemma~\ref{lem:m2}.
Applying Theorem~\ref{thm:m}, we obtain the result.
\end{proof}

\begin{rmk}
\label{rmk:m1}
Theorem~\ref{thm:m1} implies that
 if condition~(C) does not hold for $n=4$,
 then the homoclinic orbit $x^\h(t)$ persists, \ie, no bifurcation occurs,
 since $n_0=1$, as stated in Section~1.  
\end{rmk}

\begin{figure}[t]
\begin{center}
\includegraphics[scale=0.7]{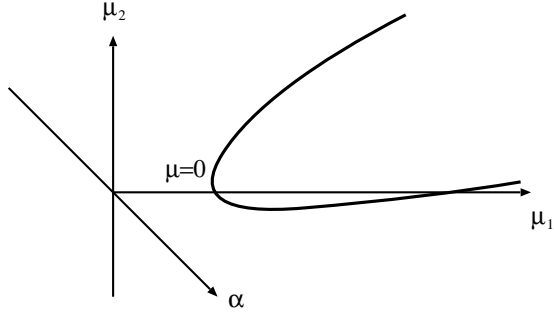}
\caption{Saddle-node bifurcation of homoclinic orbits}
\label{fig:m2}
\end{center}
\end{figure}

We next assume that $n_0=2$,
 which means that condition~(C) holds when $n=4$.
For $\mu\in\Rset^2$ define constant vectors $a_j\in\Rset^2$
 and constants $b_j\in\Rset$, $j=1,2$, as
\begin{equation}
\begin{split}
&
a_j\mu=\int_{-\infty}^{\infty}
 \langle\psi_{n_\s+j}(t),\D_\mu f(x^\h(t);0)\mu\rangle\,\d t,\\
&
b_j=\half\int_{-\infty}^{\infty}
 \langle\psi_{n_\s+j}(t),\D_x^2 f(x^\h(t);0)(\varphi_2(t),\varphi_2(t))\rangle\,\d t.
\end{split}
\label{eqn:ab}
\end{equation}
Note that $a_1=(a_{11},a_{12})$, where $a_{1j}$, $j=1,2$, are given by \eqref{eqn:a},
 as in the statement of Theorem~\ref{thm:m1}.
 
\begin{thm}
\label{thm:m2}
Under assumptions (M1)-(M3),
 let $n_0=2$ and suppose that
 the $2\times 2$ matrix $(a_1^*,a_2^*)$ is nonsingular and $(b_1,b_2)\neq(0,0)$.
Then for some open interval $I$ including $0$,
 there exists a differentiable function $\phi:I\rightarrow\Rset^2$
 with $\phi(0)=0$, $\phi'(0)=0$ and $\phi''(0)\neq 0$,
 such that a homoclinic orbit exists for $\mu=\phi(\alpha)$,
 \ie, a saddle-node bifurcation of homoclinic orbits occurs at $\mu=0$
 (see Fig.~\ref{fig:m2}).
\end{thm}

\begin{proof}
Differentiating \eqref{eqn:F1a} with respect to $\alpha$
 and using Lemma~\ref{lem:m2}, we have
\[
\D_\alpha(\id-\Pi)F(\bar{z};0;0)
 =\frac{\d}{\d t}\D_\alpha\bar{z}-\D_x f(x^\h(t);0)\D_\alpha\bar{z}=0
\]
at $(\alpha,\mu)=(0,0)$, \ie,
 $\D_\alpha\bar{z}(0;0)(t)$
 is a solution of \eqref{eqn:ve},
 so that $\D_\alpha\bar{z}(0;0)(t)=0$ by Lemma~\ref{lem:m3}.  
Using this fact and Lemma~\ref{lem:m2}, we compute \eqref{eqn:bF} as
\begin{align*}
\bar{F}_j(\alpha,\mu)
=&\int_{-\infty}^\infty\langle\psi_{n_\s+j}(t),-\D_\mu f(x^\h(t);0)\mu\\
& \qquad
 -\half\alpha^2\D_x^2 f(x^\h(t);0)(\varphi_2(t),\varphi_2(t))\rangle\d t
 +\O(\sqrt{\alpha^6+|\mu|^4})\\
=&-a_j\mu-b_j\alpha^2+\O(\sqrt{\alpha^6+|\mu|^4}),
\end{align*}
as in the proof of Theorem~\ref{thm:m1}.
Since $\bar{F}(0,0)=0$ and $|\D_\mu\bar{F}(0,0)|\neq 0$,
 we apply the implicit function theorem to show that
 there exist an open interval $I\ni 0$
 and a differentiable function $\bar{\phi}:I\rightarrow\Rset^2$
 such that $\bar{F}(\bar{\phi}(\alpha),\alpha)=0$ for $\alpha\in I$
 with $\bar{\phi}(0)=0$, $\bar{\phi}'(0)=0$ and 
\[
\bar{\phi}''(0)=-
\begin{pmatrix}
a_1\\
a_2
\end{pmatrix}^{-1}\!\!\!
\begin{pmatrix}
b_1\\
b_2
\end{pmatrix}.
\]
This implies the result along with Theorem~\ref{thm:m}.
\end{proof}

\subsection{$\Zset_2$-equivalent case}

We next consider the $\Zset_2$-equivalent case for $n_0=2$,
 and assume the following.
\begin{enumerate}
\item[\bf (M4)]
Eq.~(\ref{eqn:sys}) is {\em $\Zset_2$-equivalent}, \ie,
 there exists an $n\times n$ matrix $S$
 such that $S^2=\id_n$ and $Sf(x;\mu)=f(Sx;\mu)$.
\end{enumerate}
It follows from assumption (M4) that
 if $x=\bar{x}(t)$ is a solution to (\ref{eqn:sys}),
 then $x=S\bar{x}(t)$ is so.
We say that
 the pair $\bar{x}(t)$ and $S\bar{x}(t)$ are {\em $S$-conjugate}
 if $\bar{x}(t)\neq S\bar{x}(t)$.
See, \eg, Section~7.4 of \cite{K98}
 for more details on $\Zset_2$-equivalent systems.
In particular,
 the space $\Rset^n$ can be decomposed into a direct sum,
\[
\Rset^n=X^+\oplus X^-,
\]
where $Sx=x$ for $x\in X^+$ and $Sx=-x$ for $x\in X^-$.
Under assumption~(M4), if $x^\h(t)\in X^+$ for any $t\in\Rset$,
 then Eq.~(\ref{eqn:ve}) is also $\Zset_2$-equivalent about $\xi$
 since $S\D_x f(x;\mu)=\D_x f(Sx;\mu)S$ in general.
Here we need the following assumption.
\begin{enumerate}
\item[\bf (M5)]
For every $t\in\Rset$,
 $x^\h(t),\psi_{n_\s+1}(t)\in X^+$
 and $\varphi_2(t),\psi_{n_\s+2}(t)\in X^-$.
\end{enumerate}
Assumption~(M5) also means that $\varphi_1(t)\in X^+$.
Moreover,
 a homoclinic orbit of the form (\ref{eqn:ho}) has an $S$-conjugate counterpart
 for $\alpha\neq 0$ since it is not included in $X^+$.
Actually,
 we cannot apply Theorem~\ref{thm:m2} in this case as follows.

\begin{lem}
\label{lem:m4}
Under assumptions (M1)-(M5), we have
\[
\D_\mu f(x^\h(t);0)\mu,
 \D_x^2 f(x^\h(t);0)(\varphi_2(t),\varphi_2(t))\in X^+
\]
and
\[
\D_\mu\D_x f(x^\h(t);0)(\mu,\varphi_2(t)),
 \D_x^3 f(x^\h(t);0)(\varphi_2(t),\varphi_2(t),\varphi_2(t))\in X^-
\]
for any $t\in\Rset$ and $\mu\in\Rset^2$.
In particular, $a_2=0$ and $b_2=0$.
\end{lem}

\begin{proof}
By the $\Zset_2$-equivalence of (\ref{eqn:sys}) we compute
\[
\D_\mu f(x;0)=S\,\D_\mu f(Sx;0),\quad
\D_x^2 f(x;0)(\xi,\xi)=S\,\D_x^2 f(Sx;0)(S\xi,S\xi)
\]
and
\begin{align*}
& \D_x\D_\mu f(x;0)(\xi,\mu)
 =S\,\D_x\D_\mu f(Sx;0)(S\xi,\mu),\\
& \D_x^3 f(x;0)(\xi,\xi,\xi)
 =S\,\D_x^2 f(Sx;0)(S\xi,S\xi,S\xi).
\end{align*}
Substituting $u=x^\h(t)$ and $v=\varphi_2(t)$ into the above relations
 and using the condition that $x^\h(t)\in X^+$ and $\varphi_2(t)\in X^-$,
 we have
\begin{align*}
& \D_\mu f(x^\h(t);0)=S\D_\mu f(x^\h(t);0),\\
& \D_x^2 f(x^\h(t);0)(\varphi_2(t),\varphi_2(t))
 =S\,\D_x^2 f(x^\h(t);0)(\varphi_2(t),\varphi_2(t))
\end{align*}
and
\begin{align*}
& \D_x\D_\mu f(x^\h(t);0)(\varphi_2(t),\mu)
 =-S\D_\mu f(x^\h(t);0)(\varphi_2(t),\mu),\\
& \D_x^3 f(x^\h(t);0)(\varphi_2(t),\varphi_2(t),\varphi_2(t))
 =-S\,\D_x^2 f(x^\h(t);0)(\varphi_2(t),\varphi_2(t),\varphi_2(t)),
\end{align*}
which implies the statement.
\end{proof}

Let $\xi=\bar{\xi}^\mu(t)$ and $\bar{\xi}^\alpha(t)$ be,
 respectively, unique solutions to
\[
\dot{\xi}=\D_x f(x^\h(t);0)\xi
 +(\id-\Pi)\D_\mu f(x^\h(t);0)\mu
\]
and
\[
\dot{\xi}=\D_x f(x^\h(t);0)\xi
 +\half(\id-\Pi)\D_x^2 f(x^\h(t);0)(\varphi_2(t),\varphi_2(t))
\]
in $\bar{\Z}^1$.
It is guaranteed by Lemma~\ref{lem:m3} and \eqref{eqn:id-Pi}
 that these solutions exist.
Denote
\[
\bar{G}_\mu(t)=\D_\mu f(x^\h(t);0)\mu,\quad
\bar{G}_\alpha(t)=\half\D_x^2 f(x^\h(t);0)(\varphi_2(t),\varphi_2(t)).
\]
Using the fundamental matrices to \eqref{eqn:ve} and \eqref{eqn:ave}, 
 $\Phi(t)$ and $\Psi(t)$, and noting that $\Psi(t)=(\Phi^*(t))^{-1}$, we obtain
\[
\bar{\xi}^{\mu,\alpha}(t)
 =\Phi(t)\left[\int_0^t\Psi^*(\tau)(\id-\Pi)\bar{G}_{\mu,\alpha}(\tau)\,\d\tau
 +\Psi^*(0)\bar{\xi}_0^{\mu,\alpha}\right],
\]
where $\bar{\xi}_0^{\mu,\alpha}\in\Rset^n$ are constant vectors
 such that $\langle\psi_j(0),\bar{\xi}_0^{\mu,\alpha}\rangle=0$, $j=1,2$,
 and $\bar{\xi}^{\mu,\alpha}(t)$ are bounded on $(-\infty,\infty)$, \ie,
 one can write $\bar{\xi}_0^{\mu,\alpha}=-\Phi(0)\bar{\Xi}^{\mu,\alpha}$
 with $\bar{\Xi}^{\mu,\alpha}
 =(\bar{\Xi}_1^{\mu,\alpha},\ldots,\bar{\Xi}_n^{\mu,\alpha})$ given by
\[
\bar{\Xi}_j^{\mu,\alpha}=
\begin{cases}
\displaystyle
\int_0^{-\infty}\langle\psi_j(t),(\id-\Pi)\bar{G}_{\mu,\alpha}(t)\rangle\,\d t
 & \mbox{for $j=3,\ldots,n_\s$;}\\[2ex]
\displaystyle
\int_0^{\infty}\langle\psi_j(t),(\id-\Pi)\bar{G}_{\mu,\alpha}(t)\rangle\,\d t
 & \mbox{for $j=n_\s+3,\ldots,n$;}\\[2ex]
0 & \mbox{otherwise.}
\end{cases}
\]
Let
\begin{equation}
\begin{split}
\bar{a}_2\mu =&
\int_{-\infty}^{\infty}
 \langle\psi_{n_\s+2}(t),
 \D_\mu\D_x f(x^\h(t);0)(\mu,\varphi_2(t))\\
& \qquad
 +\D_x^2 f(x^\h(t);0)(\bar{\xi}^\mu(t),\varphi_2(t))\rangle\,\d t,\\
\bar{b}_2 =& \int_{-\infty}^{\infty}
 \langle\psi_{n_\s+2}(t),
 \sixth\D_x^3 f(x^\h(t);0)(\varphi_2(t),\varphi_2(t),\varphi_2(t))\\
&\qquad
 +\D_x^2 f(x^\h(t);0)(\bar{\xi}^\alpha(t),\varphi_2(t))\rangle\,\d t.
\end{split}
\label{eqn:barab}
\end{equation}

\begin{figure}[t]
\begin{center}
\includegraphics[scale=0.7]{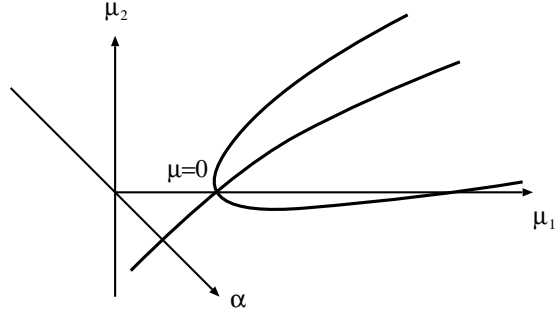}
\caption{Pitchfork bifurcation of homoclinic orbits}
\label{fig:m3}
\end{center}
\end{figure}

\begin{thm}
\label{thm:m3}
Under assumptions (M1)-(M5), suppose that condition~(C) holds,
 the $2\times 2$ matrix $(a_1^*,\bar{a}_2^*)$ is nonsingular
 and $(b_1,\bar{b}_2)\neq(0,0)$.
Then for some open interval $I\ni 0$
 there exist differentiable functions
 $\phi_j:I\rightarrow\Rset$ with $\phi_j(0)=0$, $j=1$ or 2,
 and $\phi:I\rightarrow\Rset^2$ with $\phi(0)=0$, $\phi''(0)\neq 0$
 and $\phi(\alpha)=\phi(-\alpha)$ for $\alpha\in I$, such that
 a homoclinic orbit exists on $X^+$
 for $\mu_1=\phi_1(\mu_2)$ or $\mu_2=\phi_2(\mu_1)$;
 and an $S$-conjugate pair of homoclinic orbits exist for $\mu=\phi(\alpha)$:
 a pitchfork bifurcation of homoclinic orbits occurs (see Fig.~\ref{fig:m3}).
\end{thm}

\begin{proof}
As in the proof of Theorem~\ref{thm:m2},
 we use Lemma~\ref{lem:m2} to obtain
\begin{align*}
&
\D_\mu(\id-\Pi)F(\bar{z}_0;0;0)\\
&
=\frac{\d}{\d t}\D_\mu\bar{z}_0-\D_x f(x^\h(t);0)\D_\mu\bar{z}
 -(\id-\Pi)\D_\mu f(x^\h(t);0)=0,\\
&
\D_\alpha^2(\id-\Pi)F(\bar{z}_0;0;0)\\
&
=\frac{\d}{\d t}\D_\alpha^2\bar{z}_0-\D_x f(x^\h(t);0)\D_\alpha^2\bar{z}
 -\half(\id-\Pi)\D_x^2 f(x^\h(t);0)(\varphi_2(t),\varphi_2(t))=0,
\end{align*}
where $\bar{z}_0=\bar{z}(0,0)$.
Hence, we have $(\D_\mu\bar{z}_0)\mu=\bar{\xi}^\mu$
 and $\D_\alpha^2\bar{z}_0=\bar{\xi}^\alpha$.
Noting that $\D_\alpha\bar{z}_0=0$ and using Lemma~\ref{lem:m2},
 we compute \eqref{eqn:bF} as
\begin{align*}
&
\bar{F}_j(\alpha,\mu)\\
&
=\int_{-\infty}^\infty\langle\psi_{n_\s+j}(t),
 -\D_\mu f(x^\h(t);0)\mu-\alpha\D_\mu\D_x f(x^\h(t);0)(\mu,\varphi_2(t))\\
& \qquad
 -\half\alpha^2\D_x^2 f(x^\h(t);0)(\varphi_2(t),\varphi_2(t))
 -\sixth\alpha^3\D_x^3 f(x^\h(t);0)(\varphi_2(t),\varphi_2(t),\varphi_2(t))\\
& \qquad
 -\alpha\D_x^2 f(x^\h(t);0)(\D_\mu\bar{z}(t)\mu+\D_\alpha^2\bar{z}(t)\alpha^2,
 \varphi_2(t))\rangle\d t+\O(\sqrt{\alpha^8+|\mu|^4}),
\end{align*}
so that
\begin{align*}
&
\bar{F}_1(\alpha,\mu)=a_1\mu+b_1\alpha^2+O(\sqrt{\alpha^6+|\mu|^3}),\\
&
\bar{F}_2(\alpha,\mu)
 =\bar{a}_2\alpha\mu+\bar{b}_2\alpha^3+O(\sqrt{\alpha^8+|\mu|^4}).
\end{align*}
Applying Theorem~\ref{thm:m}, we see that a unique homoclinic orbits exists
 near $(\alpha,\mu)$ satisfying
\[
\alpha=0,\ a_1\mu=0\quad\mbox{or}\quad
\mu=-
\begin{pmatrix}
a_1\\
\bar{a}_2
\end{pmatrix}^{-1}\!\!\!
\begin{pmatrix}
b_1\\
\bar{b}_2
\end{pmatrix}
\alpha^2.
\]
Note that if there exists a homoclinic orbit which is not included in $X^+$,
 then an $S$-conjugate homoclinic orbit must exist.
Thus, we repeat arguments given in the proofs
 of Theorems~\ref{thm:m1} and \ref{thm:m2} to obtain the result.
\end{proof}

\begin{rmk}
\label{rmk:m2}
From Theorems~\ref{thm:m2} and \ref{thm:m3}
 we see that if condition~(C) holds for $n=4$,
 then a saddle-node or pitchfork bifurcation occurs
 under some nondegenerate condition, since $n_0=2$, as stated in Section~1.  
\end{rmk}

\begin{rmk}
\label{rmk:m3}
Suppose that $a_1\mu+b_1\alpha^2=0$.
Then
\[
\bar{G}(t;\mu,\alpha)=\D_\mu f(x^\h(t);0)\mu
 +\half\alpha^2\D_x^2 f(x^\h(t);0)(\varphi_2(t),\varphi_2(t))
\]
satisfies $\Pi \bar{G}(t;\mu,\alpha)=0$,
 so that $\bar{\xi}(t)=\bar{\xi}^\mu(t)+\alpha^2\bar{\xi}^\alpha(t)$
is represented as
\[
\bar{\xi}(t)
 =\Phi(t)\left[\int_0^t\Psi^*(\tau)\bar{G}(\tau;\mu,\alpha)\,\d\tau-\bar{\Xi}(\mu,\alpha)\right],
\]
where $\bar{\Xi}(\mu,\alpha)
 =(\bar{\Xi}_1(\mu,\alpha),\ldots,\bar{\Xi}_n(\mu,\alpha))$ is given by
\[
\bar{\Xi}_j(\mu,\alpha)=
\begin{cases}
\displaystyle
\int_0^{-\infty}\langle\psi_j(t),\bar{G}(t;\mu,\alpha)\rangle\,\d t
 & \mbox{for $j=3,\ldots,n_\s$;}\\[2ex]
\displaystyle
\int_0^{\infty}\langle\psi_j(t),\bar{G}(t;\mu,\alpha)\rangle\,\d t
 & \mbox{for $j=n_\s+3,\ldots,n$;}\\[2ex]
0 & \mbox{otherwise.}
\end{cases}
\]
Thus, we have
\begin{align}
\bar{a}_2\mu+\bar{b}_2\alpha^2
=\int_{-\infty}^{\infty}\langle\psi_{n_\s+2}(t),
& \D_\mu\D_x f(x^\h(t);0)(\mu,\varphi_2(t))\notag\\
& +\sixth\alpha^2\D_x^3 f(x^\h(t);0)(\varphi_2(t),\varphi_2(t),\varphi_2(t))\notag\\
& +\D_x^2 f(x^\h(t);0)(\bar{\xi}(t),\varphi_2(t))\rangle\,\d t.
\label{eqn:barab1}
\end{align}
\end{rmk}

\subsection{Hamiltonian case}
We finally consider the Hamiltonian case
 and set $n=2\bar{n}$, $\bar{n}\in\Nset$, and $m=1$.
We assume the following.
\begin{enumerate}
\item[\bf (M6)]
There exists an analytic function
 $H:\Rset^n\times\Rset\rightarrow\Rset$ such that
\[
f(x;\mu)=J_n(\D_x H(x;\mu))^*, 
\]
where $J$ is the $n\times n$ symplectic matrix
\[
J_n=
\begin{pmatrix}
0 & \id_{\bar{n}}\\
-\id_{\bar{n}} & 0
\end{pmatrix}.
\]
\end{enumerate}
Assumption~(M6) means that
 Eq.~\eqref{eqn:sys} becomes a Hamiltonian system

\begin{equation}
\dot{x}=J_n(\D_x H(x;\mu))^*
\label{eqn:Hsys}
\end{equation}
with a Hamiltonian function $H(x;\mu)$.
Under this assumption we have $n_\s=n_\u=\bar{n}$.
See, \eg, \cite{MHO08}
 for more details on Hamiltonian systems.
The $\bar{n}$-dimensional stable and unstable manifolds,
 $W^\s(0)$ and $W^\u(0)$, generically intersect
 along a one-dimensional curve, \ie, a homoclinic orbit,
 since they are included in an $(n-1)$-dimensional level set
 $\{x\in\Rset^n\,|\,H(x;\mu)=H(0;\mu)\}$.

For the Hamiltonian system \eqref{eqn:Hsys},
 the VE~\eqref{eqn:ve} and its adjoint equation \eqref{eqn:ave} become
\begin{equation}
\dot{\xi}=J_n\D_x^2H(x^\h(t);0)\xi
\label{eqn:Hve}
\end{equation}
and
\begin{equation}
\dot{\xi}=\D_x^2 H(x^\h(t);0)J_n\xi,
\label{eqn:Have}
\end{equation}
respectively, since $\D_x^2H(x;0)$ is a symmetric matrix and $J_n^*=-J_n$.
Using solutions of the VE~\eqref{eqn:Hve},
 we can easily obtain solutions of the adjoint equation \eqref{eqn:Have}
 as follows.

\begin{lem}
\label{lem:m5}
Let $\xi=\varphi(t)$ be a solution of \eqref{eqn:Hve}.
Then $\xi=J_n\varphi(t)$ is a solution of \eqref{eqn:Have}.
\end{lem}

\begin{proof}
Since $\xi=\varphi(t)$ satisfies \eqref{eqn:Hve}, we have
\[
\dot{\varphi}(t)=J_n\D_x^2H(x^\h(t);0)\varphi(t).
\]
We use the equality $J_n^2=-\id_n$ to modify the above equation as
\[
J_n\dot{\varphi}(t)=\D_x^2H(x^\h(t);0)J_n\left(J_n\varphi(t)\right),
\]
which means that $J_n\varphi(t)$  is a solution of \eqref{eqn:Have}.
\end{proof}

It is clear that these solutions are orthogonal
\[
\langle\varphi(t),J_n\varphi(t)\rangle=0.
\]
Hence, we take
\begin{equation}
\psi_{\bar{n}+j}(t)=-J_n\varphi_j(t),\quad
j=1,\ldots,n_0,
\label{eqn:psij}
\end{equation}
and in particular
\begin{equation}
\psi_{\bar{n}+1}(t)=\D_x H(x^\h(t);0)^*
\label{eqn:psi1}
\end{equation}
since $\varphi_1(t)=\dot{x}^\h(t)=J_n\D_x H(x^\h(t);0)^*$.

We now turn to the problem of bifurcations in \eqref{eqn:Hsys}. 
Introduce a new parameter $\nu_1$ to modify \eqref{eqn:Hsys} as
\begin{equation}
\dot{x}=J_n\D_x H(x;\mu)^*+\nu_1\D_x H(x;\mu)^*,
\label{eqn:sys1}
\end{equation}
which depends on a two-dimensional parameter vector $\nu=(\nu_1,\mu)\in\Rset^2$.

\begin{lem}
\label{lem:m6}
For $\nu_1\ne 0$,
 Eq.~\eqref{eqn:sys1} has no homoclinic orbit
 passing through the set $\{x\in\Rset^n\,|\,\D_x H(x;\mu)\neq 0\}$,
 especially in a neighborhood of $x^\h(t)$ near $\mu=0$. 
\end{lem}

\begin{proof}
Assume that $\nu_1\neq 0$
 and let $\tilde{x}(t)$ denote a homoclinic orbit in \eqref{eqn:sys1}.
We easily compute
\[
\frac{\d}{\d t}H(\tilde{x}(t);\mu)
 =\nu_1\langle\D_x H(\tilde{x}(t);\mu),\D_x H(\tilde{x}(t);\mu)\rangle,
\]
which is negative or positive for $\nu_1<0$ or $>0$ if $\D_x H(x;\mu)\neq 0$.
This contradicts the fact that
\[
\lim_{t\rightarrow-\infty}H(\tilde{x}(t);\mu)
 =\lim_{t\rightarrow+\infty}H(\tilde{x}(t);\mu)
\]
since $\tilde{x}(t)$ is a homoclinic orbit.
Thus we obtain the result.
\end{proof}
By Lemma~\ref{lem:m6},
 no homoclinic orbit can be born from $x^\h(t)$ for $\nu_1\neq 0$.

Now we apply Theorems~\ref{thm:m2} and \ref{thm:m3} to \eqref{eqn:sys1}.
We use \eqref{eqn:psi1} to have
\[
a_1=\left(
\int_{-\infty}^\infty\langle\D_x H(x^\h(t);0),\D_x H(x^\h(t);0)\rangle\d t,0
\right)
\]
and
\begin{align*}
b_1=& \half\int_{-\infty}^\infty\left\langle\D_x H(x^\h(t);0),
 \D_x^2(J_n\D_x H(x^\h(t);0))(\varphi_2(t),\varphi_2(t))\right\rangle\d t\\
=& -\half\int_{-\infty}^\infty\left\langle J_n\D_x H(x^\h(t);0),
 \D_x^2 H(x^\h(t);0)(\varphi_2(t),\varphi_2(t))\right\rangle\d t\\
=& -\half\int_{-\infty}^\infty\frac{\d}{\d t}
 \left\langle\varphi_2(t),\D_x^2 H(x^\h(t);0)\varphi_2(t)\right\rangle\d t=0
\end{align*}
since
\begin{align*}
&
\D_x^2 H(x^\h(t);0)(\dot{\varphi}_2(t),\varphi_2(t))\\
&
=\left\langle\dot{\varphi}_2(t),
 \D_x^2 H(x^\h(t);0)\varphi_2(t)\right\rangle\\
&
=\left\langle J_n\D_x^2 H(x^\h(t);0)\varphi_2(t),
 \D_x^2 H(x^\h(t);0)\varphi_2(t)\right\rangle
=0.
\end{align*}
Using Lemma~\ref{lem:m6}
 and noting that the hypotheses of these theorems hold only if $\nu_1=0$,
 we obtain the following theorems for \eqref{eqn:Hsys} with $m=1$.
\begin{thm}
\label{thm:m4}
Under assumptions~(M1)-(M3) and (M6),
 let $n_0=2$ and suppose that $a_2,b_2\neq 0$.
Then a saddle-node bifurcation of homoclinic orbits occurs at $\mu=0$.
Moreover, it is supercritical or subcritical,
 depending on whether $a_2 b_2<0$ or $>0$.
\end{thm}
\begin{thm}
\label{thm:m5}
Under assumptions~(M1)-(M6),
 let $n_0=2$ and suppose that $\bar{a}_2,\bar{b}_2\neq 0$.
Then a pitchfork bifurcation of homoclinic orbits occurs at $\mu=0$.
Moreover, it is supercritical or subcritical,
 depending on whether $\bar{a}_2\bar{b}_2<0$ or $>0$.
\end{thm}
In these statements
 the constants $a_2,b_2,\bar{a}_2,\bar{b}_2$ are obtained
 by \eqref{eqn:ab} and \eqref{eqn:barab1} for \eqref{eqn:Hsys} with $m=1$ as follows:
\begin{equation}
\begin{split}
a_2=&-\int_{-\infty}^{\infty}
 \langle J_n\varphi_2(t),\D_\mu f(x^\h(t);0)\rangle\,\d t,\\
b_2=&-\half\int_{-\infty}^{\infty}
 \langle J_n\varphi_2(t),\D_x^2 f(x^\h(t);0)(\varphi_2(t),\varphi_2(t))\rangle\,
 \d t,\\
\bar{a}_2=&-\int_{-\infty}^{\infty}
 \langle J_n\varphi_2(t),\D_\mu\D_x f(x^\h(t);0)\varphi_2(t)\\
&\qquad +\D_x^2 f(x^\h(t);0)(\xi^\mu(t),\varphi_2(t))\rangle\,\d t,\\
\bar{b}_2=&-\int_{-\infty}^{\infty}
 \langle J_n\varphi_2(t),
 \sixth\D_x^3 f(x^\h(t);0)(\varphi_2(t),\varphi_2(t),\varphi_2(t))\\
&\qquad +\D_x^2 f(x^\h(t);0)(\xi^\alpha(t),\varphi_2(t))\rangle\,\d t,
\end{split}
\label{eqn:Hab}
\end{equation}
where we used \eqref{eqn:psij} with $j=2$, and
\begin{equation}
\xi^{\mu,\alpha}(t)
 =\Phi(t)\left[\int_0^t\Psi^*(\tau)G_{\mu,\alpha}(\tau)\,\d\tau
 -\Xi^{\mu,\alpha}\right]
\label{eqn:Hxi}
\end{equation}
with $G_\mu(t)=\D_\mu f(x^\h(t);0)$,
 $G_\alpha(t)=\half\D_x^2 f(x^\h(t);0)(\varphi_2(t),\varphi_2(t))$ and
\[
\Xi_j^{\mu,\alpha}=
\begin{cases}
\displaystyle
\int_0^{-\infty}\langle\psi_j(t),G_{\mu,\alpha}(t)\rangle\,\d t
 & \mbox{for $j=3,\ldots,n_\s$;}\\[2ex]
\displaystyle
\int_0^{\infty}\langle\psi_j(t),G_{\mu,\alpha}(t)\rangle\,\d t
 & \mbox{for $j=n_\s+3,\ldots,n$;}\\[2ex]
0 & \mbox{otherwise.}
\end{cases}
\]
See Remark~\ref{rmk:m3}.
Note that in this case the condition $a_1\nu+b_1\alpha^2=0$
 is equivalent to $\nu_1=0$. 

\section{Differential Galois theory}

Differential Galois theory deals
 with the problem of integrability by quadratures for differential equations.
Here we briefly review a part of the differential Galois theory
 which is often referred to as the Picard-Vessiot theory
 and gives a complete framework about the integrability by quadratures
 of linear differential equations with variable coefficients. 

\subsection{Picard-Vessiot extensions}

Consider a system of abstract differential equations,
\begin{equation}\label{LinearSystem}
\dot y = Ay,\quad\quad A\in\mathrm{gl}(n,\Kset),
\end{equation}
where $\Kset$ is a differential field and
 $\mathrm{gl}(n,\Kset)$ denotes the ring of $n\times n$ matrices
 with entries in $\Kset$.
We recall that a \emph{differential field} is a field
 endowed with a derivation $\partial$,
 which is an additive endomorphism
 satisfying the Leibniz rule.
By abuse of notation we write $\dot y$ instead of $\partial y$.
The set $\mathrm{C}_{\Kset}$ of elements of $\Kset$ for which $\partial$ vanishes
 is a subfield of $\Kset$
and called the \emph{field of constants of $\Kset$}.
In our application of the theory in this paper,
 the differential field $\Kset$ is
 the field of meromorphic functions on a Riemann surface $\Gamma$
 endowed with a meromorphic vector field,
 so that the field of constants becomes
 the field of complex numbers $\mathbb C$.
A \emph{differential field extension} $\Lset\supset \Kset$
 is a field extension such that $\Lset$ is also a differential field
 and the derivations on $\Lset$ and $\Kset$ coincide on $\Kset$.

A differential field extension $\Lset\supset \Kset$
 satisfying the following conditions is called a \emph{Picard-Vessiot extension}
 for \eqref{LinearSystem}.
\begin{enumerate}
\item[\bf (PV1)]
There is a fundamental matrix $\Phi$
 of \eqref{LinearSystem} with coefficients in $\Lset$.
\item[\bf (PV2)]
The field $\Lset$ is spanned
 by $\Kset$ and entries of the fundamental matrix $\Phi$.
\item[\bf (PV3)]
The fields of constants for $\Lset$ and $\Kset$ coincide.
\end{enumerate}

The system \eqref{LinearSystem}
 admits a Picard-Vessiot extension which is unique up to isomorphism.
An algebraic construction of the Picard-Vessiot extension
was given in a general situation
 by Kolchin (see, \eg, \cite{K73}). However,
 when $\Kset$ is the field of meromorphic functions,
 such a construction is unnecessary.
Actually, we have a fundamental matrix
in some field of convergent Laurent series, 
and get the Picard-Vessiot extension just by adding convergent Laurent series
to the base field $\Kset$.
 This is the point of view originally
adopted by Picard and Vessiot (cf. \cite{K57}).

We now fix a Picard-Vessiot extension $\Lset\supset \Kset$
 and fundamental matrix $\Phi$ with coefficients in $\Lset$
 for \eqref{LinearSystem}.
Let $\sigma$ be a \emph{$\Kset$-automorphism} of $\Lset$,
 which is a field automorphism of $\Lset$
 that commutes with the derivation of $\Lset$
 and leaves $\Kset$ pointwise fixed.
Obviously, $\sigma(\Phi)$ is also a fundamental matrix of \eqref{LinearSystem}
 and consequently there is a matrix $M_\sigma$ with constant entries
 such that $\sigma(\Phi)=\Phi M_\sigma$.
This relation gives a faithful representation
 of the group of $\Kset$-automorphisms of $\Lset$
 on the general linear group as
\[
R\colon {\rm Aut}_{\Kset}(\Lset)\to {\rm GL}(n,{\rm C}_{\Lset}),
\quad \sigma\mapsto M_{\sigma},
\]
where ${\rm GL}(n,{\rm C}_{\Lset})$
is the group of $n\times n$ invertible matrices with entries in ${\rm C}_{\Lset}$.
The image of $R$
 is a linear algebraic subgroup of ${\rm GL}(n,{\rm C}_{\Lset})$,
 which is called the \emph{differential Galois group} of \eqref{LinearSystem}
and denoted by ${\rm Gal}(\Lset/\Kset)$.
This representation is not unique
 and depends on the choice of the fundamental matrix $\Phi$,
 but a different fundamental matrix only gives rise to a conjugated representation.
Thus, the differential Galois group is unique up to conjugation
 as an algebraic subgroup of the general linear group.

\begin{defn}
A differential field extension $\Lset\supset\Kset$ is called 
\begin{enumerate}
\item[(i)]
an \emph{integral extension} if there exists $a\in\Lset$ such
that $\dot a \in \Kset$ and $\Lset = \Kset(a)$,
 where $\Kset(a)$ is the smallest extension of $\Kset$ containing $a$;
\item[(ii)]
an \emph{exponential extension} if there exists $a\in\Lset$ such
that $\dot{a}/a \in \Kset$ and $\Lset = \Kset(a)$;
\item[(iii)]
an \emph{algebraic extension} if there exists $a\in\Lset$
 such that it is algebraic over $\Kset$ and $\Lset = \Kset(a)$.
\end{enumerate}
\end{defn}

\begin{defn}
A differential field extension $\Lset\supset\Kset$ is called a 
\emph{Liouvillian extension} if it can be decomposed as a tower of extensions,
$$
\Lset = \Kset_n \supset \ldots \supset \Kset_1\supset 
\Kset_0 = \Kset,
$$
such that each extension $\Kset_{i+1}\supset \Kset_i$ is
either integral, exponential or algebraic. It is called \emph{strictly Liouvillian}
if in the tower only integral and exponential extensions appear. 
\end{defn}

Let $G\subset {\rm GL}(n,{\rm C}_{\Lset})$ be an algebraic group.
Then it contains a unique maximal connected algebraic subgroup $G^0$,
 which is called the \emph{connected component of the identity}
 or \emph{connected identity component}.
The connected identity component $G^0\subset G$
 is a normal algebraic subgroup and the smallest subgroup of finite index,
 \ie, the quotient group $G/G^0$ is finite.
By the Lie-Kolchin Theorem \cite{K57,VS03},
a connected solvable linear algebraic group is triangularizable.
Here
 a subgroup of ${\rm GL}(n,{\rm C}_{\Lset})$ is
said to be \emph{triangularizable}
 if it is conjugated to a subgroup of the group of upper triangular matrices.
The following theorem relates
 the solvability and triangularizability of the differential Galois group
 with the (strictly) Liouvillian Picard-Vessiot extension
 (see \cite{K57,VS03} and \cite{BM09} for the proofs of the first and second parts,
 respectively).

\begin{thm}
\label{thm:dg}
Let $\Lset\supset\Kset$ be a Picard-Vessiot extension of \eqref{LinearSystem}.
\begin{enumerate}
\item[(i)]
The connected identity component
 of the differential Galois group ${\rm Gal}(\Kset/\Lset)$ is solvable
 if and only if $\Lset\supset\Kset$ is a Liouvillian extension.
\item[(ii)]
If the differential Galois group ${\rm Gal}(\Kset/\Lset)$ is triangularizable,
 then $\Lset\supset\Kset$ is a strictly Liouvillian extension. 
\end{enumerate}
\end{thm}

\subsection{Monodromy group and Fuchsian equations}

Let $\Kset$ be the field of meromorphic functions in a Riemann surface $\Gamma$
and let $t_0\in\Gamma$
 be a nonsingular point in \eqref{LinearSystem}.
We prolong the fundamental matrix $\Phi(t)$ analytically
 along any loop $\gamma$ based at $t_0$ and containing no singular point,
 and obtain another fundamental matrix $\gamma*\Phi(t)$.
So there exists a constant nonsingular matrix $M_{[\gamma]}$ such that
$$
\gamma*\Phi(t) = \Phi(t)M_{[\gamma]}.
$$
The matrix $M_{[\gamma]}$ depends on the homotopy class $[\gamma]$
 of the loop $\gamma$
 and is called the \emph{monodromy matrix} of $[\gamma]$.

The set of singularities in \eqref{LinearSystem}
is a discrete subset of $\Gamma$, which is denoted by $\S$.
Let $\pi_1(\Gamma\setminus\S,t_0)$
 be the fundamental group of homotopy classes of loops based at $t_0$.
We have a representation
$$
\tilde{R}\colon \pi_1(\Gamma\setminus\S,t_0)\to {\rm GL}(n,\Cset), 
\quad [\gamma]\mapsto M_{[\gamma]}.
$$
The image of $\tilde{R}$ is called the \emph{monodromy group}
 of \eqref{LinearSystem}.
As in the differential Galois group,
the representation $\tilde{R}$ depends on the choice of the fundamental matrix,
but the monodromy group is defined as a group of matrices up to conjugation.
In general, monodromy transformations
 define automorphisms of the corresponding Picard-Vessiot extension.

A singular point of \eqref{LinearSystem} is called \emph{regular}
 if the growth of solutions along any ray approaching the singular point
 is bounded by a meromorphic function;
 otherwise it is called \emph{irregular}.
In particular, if $A=B(t)/t$ with $B(t)$ holomorphic at zero,
 then Eq.~\eqref{LinearSystem} has a regular singularity at $t=0$.
Eq.~\eqref{LinearSystem} is said to be \emph{Fuchsian}
 if all singularities are regular.
Any univalued solution of a Fuchsian equation is meromorphic.
This gives us the following result
 along with the normality of the Picard-Vessiot extensions
 (see, \eg, Theorem 5.8 in \cite{VS03} for the proof).

\begin{thm}[Schlessinger]\label{th:sl}
Assume that Eq.~\eqref{LinearSystem} is Fuchsian.
Then the differential Galois group of \eqref{LinearSystem}
 is the Zariski closure of the monodromy group. 
\end{thm}

Since the group of triangular matrices is algebraic,
 the Zariski closure of a triangularizable group is triangularizable.
Noting this fact,
 we obtain the following result immediately from Theorem~\ref{th:sl}.

\begin{cor}\label{TriangularG}
Assume that Eq.~\eqref{LinearSystem} is Fuchsian.
Then the monodromy group is triangularizable
 if and only if the differential Galois group is triangularizable. 
\end{cor}

\section{Proof of Theorem~\ref{thm:main}}

We now prove our main result.
Here we note that 
 $f$ and $\M$ can be extended to complex analytic ones
 in a neighborhood of $\Rset^4$ in $\Cset^4$
 since they are real analytic.
In this section
 we consider such complexifications including one of \eqref{eqn:ve}.

\subsection{Normal and tangential variational equations}

We first use assumption~(A2)
 to decompose the four-dimensional VE~\eqref{eqn:ve}
 into two-dimensional normal and tangent parts,
 so that we reduce our analysis
 to a two-dimensional system.

\begin{figure}[t]
\begin{center}
\includegraphics[scale=1]{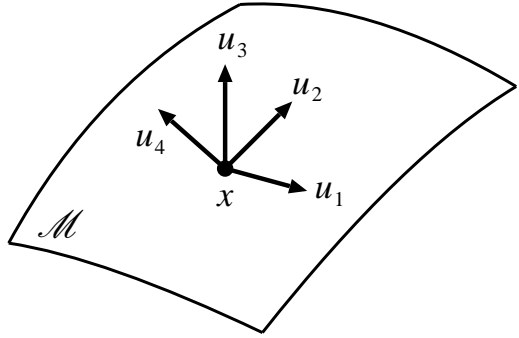}
\caption{Generators $u_j$, $j=1,2,3,4$.}
\label{fig:u}
\end{center}
\end{figure}

Consider the \emph{generic variational equation} at $\mu=0$,
\begin{equation}
\dot x=f(x;0),\quad
\dot{\xi}=\D_x f(x;0)\xi,
\label{eqn:gve}
\end{equation}
which defines a flow on the tangent bundle $T\Cset^4$
 and which is linear on its fibers.
Since $\M$ is invariant under the flow of \eqref{eqn:sys} by assumption~(A2),
 the tangent bundle $T\M$ and tangent bundle of $\Cset^4$ restricted to $\M$,
 $T\Cset^4|_\M$, are invariant under the flow of \eqref{eqn:gve}.
The normal bundle $N\M$ is identified with the quotient $T\Cset^4|_{\M}/T\M$
 by definition, and its fiber at $x\in\M$ is the linear space $T_x\Cset^4/T_x\M$. 
Let us take a \emph{moving frame} on $\M$,
 \ie, a system of generators $u_j\in\Cset^4$, $j=1,2,3,4$,
 for the tangent space $T_x\Cset^4$ with $x\in\M$,
 such that $T_x\M=\Span\{u_1,u_2\}$ and $N_x\M=\Span\{u_3,u_4\}$
 (see Fig.~\ref{fig:u}).
We introduce new coordinates $(\chi_1,\chi_2,\eta_1,\eta_2)\in\Cset^4$ by
\[
\xi=\chi_1 u_1+\chi_2 u_2+\eta_1 u_3+\eta_2 u_4
\]
on $T\Cset^4|_\M$.
The invariance of $T\M$ under the flow of \eqref{eqn:gve}
 ensures that the plane $\eta_1=\eta_2=0$ is invariant
 in the restriction of \eqref{eqn:gve} to $\M$.
So the second equation of \eqref{eqn:gve} is rewritten 
 by a block form in the new coordinates as
\begin{equation}
\begin{pmatrix} 
\dot\chi_1\\
\dot\chi_2\\
\dot\eta_1\\
\dot\eta_2
\end{pmatrix}
=\begin{pmatrix}
A_\chi(x) & A_\c(x) \\[1ex]
0 & A_\eta(x)
\end{pmatrix}
\begin{pmatrix}
\chi_1 \\ \chi_2\\ \eta_1 \\ \eta_2
\end{pmatrix}
\label{eqn:newve}
\end{equation}
if restricted to $\M$.
Here the $2\times 2$ matrix functions $A_\chi(x)$, $A_\eta(x)$ and $A_\c(x)$
 are analytic and obtained via algebraic and differential manipulation
 from $\D_x f(x;0)$ and $u_j$, $j=1,2,3,4$.

Let $e_j$ be a unit eigenvector of $\D_x f(0;0)$
 corresponding to the eigenvalue $\lambda_j$ for $j=1,2,3,4$.
From \eqref{eqn:newve} we obtain the following result.

\begin{lem}
\label{lem:p0}
The tangent space of $\M$ at the origin $x=0$ is spanned by $e_{j_-}$ and $e_{j_+}$,
 \ie, $T_0\M=\Span\{e_{j_-},e_{j_+}\}$, where $j_-=1$ or $2$ and $j_+=3$ or $4$.
\end{lem}

\begin{proof}
Let $x=0$ in \eqref{eqn:newve}.
We easily see that two eigenvalues of the Jacobian matrix
\[
\begin{pmatrix}
A_\chi(0) & A_\c(0) \\[1ex]
0 & A_\eta(0)
\end{pmatrix}
\]
are those of $A_\chi(0)$,
 and the associated eigenvectors have the form
 $\xi=(\chi_\pm,0)\in\Cset^2\times\Cset^2$,
 where $\chi_\pm$ are eigenvectors of $A_\chi(0)$.
Since $\M$ contains a homoclinic orbit,
 $A_\chi(0)$ must have a pair of positive and negative eigenvalues.
Thus we obtain the result.
\end{proof}
Henceforth we denote $\lambda_\pm=\lambda_{j_\pm}$
 and write the other eigenvalues as $\mu_\pm$, where $\mu_-<0<\mu_+$.

Let us consider \eqref{eqn:newve} for $x=x^\h(t)$.
Taking the normal components $\eta=(\eta_1,\eta_2)$,
 we have the \emph{normal variational equation (NVE)} around $x^\h(t)$,
\begin{equation}
\dot \eta = A_\eta(x^\h(t))\eta.
\label{eqn:nve}
\end{equation}
Moreover, we set $\chi=(\chi_1,\chi_2)$ and $\eta=0$ to obtain
 the \emph{tangent variational equation (TVE)} around $x^\h(x)$,
\begin{equation}\label{eqn:tve}
\dot\chi = A_\chi(x^\h(t))\chi,
\end{equation}
which governs the dynamics of \eqref{eqn:gve} on $T\M$ for $x=x^\h(t)$.
When interested in a necessary condition for condition~(C),
 we only have to deal with the two-dimensional NVE~\eqref{eqn:ve}
 instead of the four-dimensional VE~\eqref{eqn:ve} as follows.

\begin{lem}\label{lem:Cprime}
If condition~(C) holds, then
\begin{enumerate}
\item[\bf(C')]
the NVE~\eqref{eqn:nve} has a non-vanishing bounded solution. 
\end{enumerate}
\end{lem}

\begin{proof}
Since $A_\chi(0)=\lim_{t\to\pm\infty}A_\chi(x^\h(t))$
 has two eigenvalue $\lambda_-<0<\lambda_+$ (cf. the proof of Lemma~\ref{lem:p0}),
 we see that the TVE~\eqref{eqn:tve} has one independent bounded solution at most,
 as in Lemma~\ref{lem:m1}.
In addition, the solution $\xi=\dot{x}^\h(t)$ of the VE~\eqref{eqn:ve}
 gives a bounded solution of the TVE~\eqref{eqn:tve}.
Hence, the TVE~\eqref{eqn:tve} has only one independent bounded solution.

Let us assume that condition~(C) holds,
 \ie, the VE~\eqref{eqn:ve} has two independent bounded solutions.
Since the TVE~\eqref{eqn:tve} has only one independent bounded solution,
 the NVE~\eqref{eqn:nve} must have a non-vanishing bounded solution.
\end{proof}

\subsection{Analyses of the VE, NVE and TVE}

We next analyze
 the VE~\eqref{eqn:ve}, NVE~\eqref{eqn:nve} and TVE~\eqref{eqn:tve}
 by using adequate systems of coordinates.
We begin with the following lemma.
\begin{lem}
\label{lem:p1}
There exist analytical functions $h_\pm: U\rightarrow\Cset^4$
 defined in a neighborhood $U$ of $x=0$ in $\Cset$ such that $h_\pm(0)=0$
 and the complexification of the homoclinic orbit $x^\h(t)$ is represented as
\[
x^\h(t)=
\begin{cases}
h_+\left(\e^{\lambda_+ t}\right)\quad\mbox{for $\Re(t)>0$};\\
h_-\left(\e^{\lambda_- t}\right)\quad\mbox{for $\Re(t)<0$}
\end{cases}
\]
in $U$.
\end{lem}

\begin{proof}
Let $\Gamma$ denote the complex separatrix of $x=0$
 on the invariant manifold $\M$ in \eqref{eqn:sys}.
It follows from assumption~(A2) that $x=0$ is a double point on $\Gamma$.
Hence, the intersection of $\Gamma\setminus\{0\}$
 with a small neighborhood of $x=0$ in $\mathbb C^4$
 consists of two pointed disks $\Sigma_{\pm}^*$.
We add two different copies of the origin to these pointed disks,
 and get two disks $\Sigma_{\pm}$.
Using Lemma~\ref{lem:p0} and holomorphic changes of coordinates,
 we represent the restriction of \eqref{eqn:sys} to $\Sigma_{\pm}$ as
\begin{equation}\label{eqn:disk}
\dot z_{\pm} = \lambda_{\pm}z_{\pm} + g_{\pm}(z_{\pm})z_{\pm}^2,
\end{equation}
where $z_{\pm}$ represent coordinates in the disks $\Sigma_{\pm}$
 and $g_{\pm}(z)$ are holomorphic functions.
Since $\Re(\lambda_{\pm})\neq 0$,
 Eq.~\eqref{eqn:disk} is analytically linearizable 
(see, \eg, Section~1.2 in Chapter~4 of \cite{AI88})
and hence rewritten as
\begin{equation}\label{eqn:disk1}
\dot \zeta_{\pm} = \lambda_{\pm}\zeta_{\pm}
\end{equation}
under holomorphic changes of coordinates,
 by reducing the sizes of $\Sigma_{\pm}$ if necessary.
The flow of \eqref{eqn:disk1} is given by $\zeta_{\pm}(t) = e^{\lambda_{\pm}t}$
 and the functions $h_\pm$ are the inverses of the transformation
 $x\mapsto\zeta_\pm$ in $U$.
\end{proof}

\begin{figure}
\begin{center}
\includegraphics[scale=0.9]{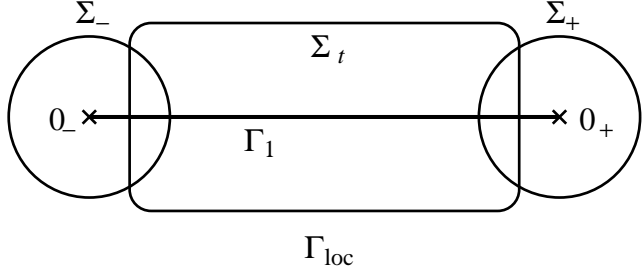}
\caption{Local Riemann surface $\Gamma_\loc$
 and its covering $\{\Sigma_\pm,\Sigma_t\}$.}
\label{fig:Gamma}
\end{center}
\end{figure}

Let $\Gamma_{0}=\{x=x^{\h}(t)\,|\,t\in\Rset\}\cup\{0\}$.
The curve $\Gamma_{0}$ in the complex space $\Cset^{4}$
consists of the homoclinic orbit $x^{\h}(t)$ and saddle $x=0$,
 which is a double point.
We introduce two points $0^{+}$ and $0^{-}$
corresponding to the origin
for desingularizing the curve $\Gamma_{0}$.
The points $0_{+}$ and $0_{-}$
are represented in the temporal parametrization
by $t=+\infty$ and $t=-\infty$, respectively.
Take a sufficiently narrow, simply connected neighborhood $\Sigma_t$
 of $\Gamma_0\setminus (\Sigma_+\cup\Sigma_-)$ in $\Gamma$,
such that it contains no singularity of \eqref{eqn:ve}
 and intersects $\Sigma_{\pm}$ in two simply connected domains.
We set $\Gamma_{\loc}=\Sigma_{-}\cup\Sigma_t\cup\Sigma_+$,
 so that $\Gamma_{\loc}$ is simply connected
 and contains only two singularities of the VE~\eqref{eqn:ve} at $0_\pm$.
See Fig.~\ref{fig:Gamma}.

Using Lemma~\ref{lem:p1}
 and the covering $\{\Sigma_\pm,\Sigma_t\}$ of $\Gamma_{\loc}$,
 we introduce three charts on the Riemann surface $\Gamma_{\loc}$ as follows:
The charts in $\Sigma_\pm$ are given by $z=\e^{\lambda_\pm t}$
 and the chart in $\Sigma_{t}$ is given by $t$.
Thus, we transform the VE~\eqref{eqn:ve} onto $\Gamma_\loc$,
 such that it is unchanged in $\Sigma_{t}$ and written as
\begin{equation}
\frac{\d\xi}{\d z_\pm}
 =\frac{1}{\lambda_\pm z_\pm}\D_x f(h_\pm(z_\pm);0)\,\xi
\label{eqn:ve0}
\end{equation}
in $\Sigma_{\pm}$.
We easily see see that $\D_x f(\chi_\pm(z);0)$ is analytic in $z$, and
 obtain the following lemma.

\begin{lem}
\label{lem:p3}
The singularities of the VE~\eqref{eqn:ve} at $0_\pm$ are regular.
\end{lem}

Thus, 
the VE~\eqref{eqn:ve} is Fuchsian on $\Gamma_{\loc}$,
 and so are the NVE~\eqref{eqn:nve} and TVE~\eqref{eqn:tve}.
We also have the following result.

\begin{lem}
\label{lem:p4}
By the coordinates $z_\pm$,
 the NVE~\eqref{eqn:nve} and TVE~\eqref{eqn:tve} are,
 respectively, rewritten as
$$
\frac{\d\eta}{\d z_\pm}=\frac{1}{z_\pm}B_\eta^\pm(z_\pm)\eta
\quad\mbox{and}\quad
\frac{\d\chi}{\d z_\pm}=\frac{1}{z_\pm}B_\chi^\pm(z_\pm)\chi
$$
in $\Sigma_{\pm}$,
 where the $2\times 2$ matrix functions $B_\eta^\pm(z)$ and $B_\chi^\pm(z)$
 are holomorphic and
$$
B_\eta^\pm(0)=
\begin{pmatrix}
\mu_+/\lambda_\pm & 0\\
0 & \mu_-/\lambda_\pm 
\end{pmatrix},\quad
B_\chi^\pm(0)=
\begin{pmatrix}
\lambda_+/\lambda_\pm & 0\\
0 & \lambda_-/\lambda_\pm 
\end{pmatrix}.
$$
\end{lem}

\begin{proof}
As in \eqref{eqn:ve0},
 we can rewrite the NVE~\eqref{eqn:nve} and TVE~\eqref{eqn:tve} as
\[
\frac{\d\eta}{\d z_\pm}=\frac{1}{\lambda_\pm z_\pm}A_\eta(h_\pm(z_\pm))\eta
\quad\mbox{and}\quad
\frac{\d\chi}{\d z_\pm}=\frac{1}{\lambda_\pm z_\pm}A_\chi(h_\pm(z_\pm))\chi,
\]
respectively, in $\Sigma_\pm$.
Noting that $A_\eta(0)$ and $A_\chi(0)$
 have eigenvalues $\lambda_\pm$ and $\mu_\pm$,
 we prove the result.
\end{proof}

\subsection{Application of the differential Galois theory}

We apply the differential Galois theory of Section~3
 to the NVE~\eqref{eqn:nve} and VE~\eqref{eqn:ve}
 on the Riemann surface $\Gamma_\loc$.
To this end,
 we need an auxiliary result for two-dimensional systems of the form 
\begin{equation}\label{RegularSingular}
\frac{\d y}{\d z} =\frac{1}{z}B(z)y,\quad
y\in\Cset^2,
\end{equation}
where the $2\times 2$ matrix function $B(z)$ is holomorphic at $z=0$.
The system \eqref{RegularSingular}
 has a regular singular point at zero
 and its fundamental matrix near $z=0$ is expressed as
\begin{equation}\label{FMatrix}
\Phi(z) = Y(z)z^E,
\end{equation}
where the $2\times 2$ matrix function $Y(z)$ is meromorphic near $z=0$
 and $E$ is a constant matrix.
Using this expression,
 we can easily compute the monodromy matrix corresponding to an infinitesimal loop
 around $z=0$ as $M = \exp(2\pi i E)$
 (see \cite{B00}, page 8).

\begin{lem}\label{Fuchsian}
Assume that $B(0)$ has two real eigenvalues $\rho_\pm$
 such that $\rho_-<0<\rho_+$.
Then the following statements hold:

\begin{enumerate}
\item[(i)]
Eq.~\eqref{RegularSingular} has a solution $\bar{y}(z)$
 which is bounded along any ray approaching $z=0$
 and unique up to constant factors.
\item[(ii)]
Any other independent solutions of \eqref{RegularSingular}
 are unbounded along any ray approaching $z=0$.
\item[(iii)]
The monodromy matrix $M$ has an eigenvalue $\e^{2\pi i \rho_+}$,
 for which the associated eigenvector is given by  $\bar{y}(z)$,
$$
M\bar{y}(z) =\e^{2\pi i\rho_+}\bar{y}(z).
$$
\end{enumerate}

\end{lem}

\begin{proof}
We first consider the \emph{nonresonant} case
 of $\rho_+-\rho_-\not\in \mathbb Z$.
Using Theorem~5 in Chapter~2 of \cite{B00},
 we have $E = B(0)$ in \eqref{FMatrix}.
Under a constant linear transformation,
 we assume that $B(0)$ is diagonal and
\begin{equation}\label{NonLogarithmic}
\Phi(z) = Y(z)
\begin{pmatrix}
z^{\rho_-} & 0\\
0 & z^{\rho_+}
\end{pmatrix}.
\end{equation}
Letting $\Phi(z)=(y_1(z),y_2(z))$ and $Y(z)=(v_1(z),v_2(z))$ in \eqref{NonLogarithmic},
 we have $y_1(z) = z^{\rho_-}v_1(z)$ and $y_2(z) = z^{\rho_+} v_2(z)$.
Since $v_1(z),v_2(z)$ are holomorphic
 and $z^{\rho_+}$ (resp. $z^{\rho_-}$) is bounded (resp. unbounded)
 along any ray approaching $z=0$,
 so is the solution $y_2(z)$ (resp. $y_1(z)$).
Thus, we prove parts~(i) and (ii).
Moreover, part~(iii) immediately follows
 by noting that $M=\exp(2\pi i B(0))$ is also diagonal in the present coordinates
 and $\bar{y}(z)=y_2(z)$.

We next consider the \emph{resonant} case of $\rho_+-\rho_-\in \mathbb Z$. 
Let $\rho_0\in[0,1)$ be a real number
 such that $\rho_\pm=\rho_0\pm m_\pm$ for some positive integers $m_\pm$.
Using Theorem~6 in Chapter~2 of \cite{B00},
 we have a fundamental matrix of the form
\begin{equation}\label{Resonant}
\Phi(z) = Y(z)
\begin{pmatrix}
z^{-m_-} & 0\\
0 & z^{m_+}
\end{pmatrix}
z^{E'}
\end{equation}
under a constant linear transformation,
 where $E'$ is a $2\times 2$ matrix of the Jordan form
 having a double eigenvalue $\rho$.
If $E'$ is diagonal,
 then the expression \eqref{Resonant} coincides with \eqref{NonLogarithmic},
 so that we can prove parts~(i), (ii) and (iii) as above. 

Suppose that $E'$ is nondiagonal.
This situation corresponds to so-called \emph{logarithmic singularity}.
Eq.~\eqref{Resonant} becomes
\begin{equation}\label{Logarithmic}
\Phi(z) = Y(z)
\begin{pmatrix}
z^{\rho_-} & 0\\
z^{\rho_-}\log z & z^{\rho_+}
\end{pmatrix},
\end{equation}
and the same argument as in the nonresonant case yields parts~(i) and (ii).
Moreover, the monodromy matrix is computed as
$$
M =\exp
\begin{pmatrix}
2\pi i \rho_0 & 0\\
2\pi i & 2\pi i\rho_0
\end{pmatrix}
= 
\begin{pmatrix}
\e^{2\pi i \rho_0} & 0\\
2\pi i \e^{2\pi i \rho_0} & \e^{2\pi i\rho_0}
\end{pmatrix}.
$$
Letting $Y(z)=(v_1(z),v_2(z))$ as above,
 we see that $y_2(z)=\e^{2\pi i\rho_0}v_2(z)$ is an eigenvector of $M$
 for the eigenvalue $\e^{2\pi i\rho_0} = \e^{2\pi i\rho_+}$.
Thus we complete the proof.
\end{proof}

We turn to the NVE~\eqref{eqn:nve} and VE~\eqref{eqn:ve}
 on the Riemann surface $\Gamma_\loc$.

\begin{thm}
\label{thm:p}
Under condition~(C'), the monodromy group of the transformed NVE~\eqref{eqn:nve}
 on $\Gamma_\loc$ is triangularizable. 
\end{thm}

\begin{proof}
Suppose that condition~(C') holds, \ie,
 the NVE~\eqref{eqn:nve} has a non-vanishing bounded solution.
Using Lemmas~\ref{lem:p4} and \ref{Fuchsian}, we see that
 the corresponding bounded solution of the NVE~\eqref{eqn:nve}
 is a common eigenvector of the monodromy matrices around $0_\pm$.
Since a group of $2\times 2$ matrices with a common eigenvector
 is triangularizable, so is the monodromy group of the NVE~\eqref{eqn:nve} on $\Gamma_\loc$.
\end{proof}

\subsection*{Proof of Theorem~\ref{thm:main}}
Suppose that condition~(C) holds.
Then it follows from Lemma~\ref{lem:Cprime} that condition~(C') holds.
Using Corollary~\ref{TriangularG} and Theorem~\ref{thm:p}, we see that
 the differential Galois group of the NVE~\eqref{eqn:nve}
 on $\Gamma_\loc$ is triangularizable.
Recall that the TVE~\eqref{eqn:tve} is always integrable
 and solutions of the VE~\eqref{eqn:ve} are obtained
 by substituting solutions of the NVE~\eqref{eqn:nve}
 and by solving the resulting equation by variation of constants.
Hence, the differential Galois group of the VE~\eqref{eqn:ve} on $\Gamma_\loc$
 is triangularizable.
This completes the proof.
\qed

\begin{rmk}\label{rmk:p3}
Suppose that the full Riemann surface $\Gamma$ in $\Cset^4$
 has genus zero,
 and let $\bar\Gamma$ be its desingularization.
Then
 $\bar\Gamma$ is a Riemann sphere and hence
 the field of meromorphic functions on $\bar\Gamma$ is isomorphic
 to that of rational functions $\mathbb C(z)$.
Assume that the pullback of the VE~\eqref{eqn:ve} to $\bar\Gamma$
 has exactly three regular singularities $\{0_{\pm},p\}$.
Then $\Gamma_{\loc}\setminus\{0_{\pm}\}$ is homotopic to 
$\bar\Gamma\setminus \{0_{\pm},p\}$,
so that both Riemann surfaces give rise
 to equivalent monodromy representations.
Hence, we see via Theorems~\ref{thm:main} and \ref{th:sl} that
 when regarded as a differential equation with coefficients in $\Cset(z)$,
 the VE~\eqref{eqn:ve} is integrable
 by Liouvillian functions on $\Cset(z)$,
 which lie in a Liouvillian extension of $\Cset(z)$,
 if condition~(C) holds.
This situation happens in the example given in the next section.
\end{rmk}

\section{Example}

To illustrate our theory, we consider
\begin{equation}
\begin{split}
&
\dot{x}_1=x_3,\quad
\dot{x}_3=x_1-(x_1^2+\beta_1 x_2^2)x_1-\beta_3 x_2,\\
&
\dot{x}_2=x_4,\quad
\dot{x}_4=sx_2-(\beta_1 x_1^2+\beta_2 x_2^2)x_2-\beta_3 x_1-\beta_4 x_2^2,
\end{split}
\label{eqn:ex}
\end{equation}
which represents steady states
 in coupled real Ginzburg-Landau equations of the form
\begin{equation}
\begin{split}
&
\partial_t U_1=\partial_x^2 U_1-U_1+(U_1^2+\beta_1 U_2^2) U_1+\beta_3 U_2,\\
&
\tau\partial_t U_2=\partial_x^2 U_2 -sU_2+(\beta_1 U_1^2+\beta_2 U_2^2) U_2
 +\beta_3 U_1 +\beta_4 U_2^2,
\end{split}\quad
U_1,U_2\in\Rset,
\label{eqn:PDEs}
\end{equation}
with $x_1=U_1$, $x_2=U_2$, $x_3=\partial_x U_1$ and $x_4=\partial_x U_2$,
 where $s$, $\tau$ and $\beta_j$, $j=1,2,3,4$, are constants.
See \cite{M02} and references therein
 for general information on Ginzburg-Landau equations
 and \cite{DHV04,R99} for examples of PDEs of the type \eqref{eqn:PDEs}.
Eq.~\eqref{eqn:ex} also represents a two-mode truncation
 of non-planar vibrations of a buckled beam 
 when $\beta_3=\beta_4=0$ (see, \eg, \cite{Y99}).
Specifically, we assume that $s\ge 1$
 since the case of $s<1$ can be treated very similarly.

Eq.~\eqref{eqn:ex} is Hamiltonian with a Hamiltonian function
\[
H=\half(-x_1^2-s x_2^2+\beta_1 x_1^2 x_2^2+x_3^2+x_4^2)
 +\fourth(x_1^4+\beta_2 x_2^4)+\beta_3 x_1 x_2+\third\beta_4 x_2^3.
\]
The eigenvalues of the Jacobian matrix of the right-hand-side of \eqref{eqn:ex}
 at $x=0$ are given by
\[
\lambda_1=-\sqrt{s},\quad
\lambda_2=-1,\quad
\lambda_3=1,\quad
\lambda_4=\sqrt{s}.
\]
When $\beta_3=0$,
 the $(x_1,x_3)$- and the $(x_2,x_4)$-planes are invariant
 under the flow of \eqref{eqn:ex} and there exist a pair of homoclinic orbits
\begin{equation}
x_\pm^\h(t)
 =\left(\pm\sqrt{2}\,\sech t,0,\mp\sqrt{2}\,\sech\,t\tanh t,0\right)
\label{eqn:exhomo}
\end{equation}
on the $(x_1,x_3)$-plane.
Here we are only interested in the homoclinic orbits given by \eqref{eqn:exhomo}
 although another pair of homoclinic orbits exist on the $(x_2,x_4)$-plane.
Thus, assumptions~(A1) and (A2) (\ie, (M1) and (M2)) hold for $\beta_3=0$
 as well as assumption~(M6).

Let $\beta_3=\beta_4=0$.
Then Eq.~\eqref{eqn:ex} is $\Zset_2$-equivalent with
\[
S=
\begin{pmatrix}
1 & 0 & 0 & 0\\
0 & -1 & 0 & 0\\
0 & 0 & 1 & 0\\
0 & 0 & 0 & -1
\end{pmatrix}.
\]
Thus, assumption~(M3) holds.
We also have $X^+=\{(x_1,0,x_3,0)\in\Rset^4\,|\,x_1,x_3\in\Rset\}$,
 $X^-=\{(0,x_2,0,x_4)\in\Rset^4\,|\,x_2,x_4\in\Rset\}$,
 and $x_\pm^\h(t)\in X^+$.
Henceforth we take $x^\h(t)=x_+^\h(t)$ to avoid complication in notation.

For $\beta_3=\beta_4=0$,
 the VE around the homoclinic orbits $x^\h(t)$ is given by
{\setcounter{enumi}{\value{equation}}
\addtocounter{enumi}{1}
\setcounter{equation}{0}
\renewcommand{\theequation}{\theenumi\alph{equation}}
\begin{align}
&
\dot{\xi}_1=\xi_3,\quad
\dot{\xi}_3=(1-6\,\sech^2 t)\xi_1,\label{eqn:exve1}\\
&
\dot{\xi}_2=\xi_4,\quad
\dot{\xi}_4=(s-2\beta_1\,\sech^2 t)\xi_2,\label{eqn:exve2}
\end{align}
\setcounter{equation}{\value{enumi}}}
which is divided
 into uncoupled two-dimensional linear systems of the form
\begin{equation}
\dot{\eta}_1=\eta_2,\quad
\dot{\eta}_2=(\nu_1-\nu_2\,\sech^2 t)\eta_1.
\label{eqn:lsys}
\end{equation}

Eqs.~\eqref{eqn:exve1} and \eqref{eqn:exve2}, respectively,
 give the TVE and NVE for \eqref{eqn:ex}.
Letting $z=\sech^2 t$ in \eqref{eqn:lsys},
 we have a Fuchsian second-order differential equation
\begin{equation}
\frac{\d^2\eta}{\d z^2}+\frac{3z-2}{2z(z-1)}\frac{\d\eta}{\d z}
 +\frac{\nu_1-\nu_2 z}{4z^2(z-1)}\eta=0,
\label{eqn:2ode}
\end{equation}
which has regular singularities at $z=0,1,\infty$
 and whose solutions are expressed by a Riemann $P$ function \cite{WW27} as
\begin{equation}
P\left\{
\begin{matrix}
0 & 1 & \infty\\
s_0^- & 0 & s_\infty^- & z\\
s_0^+ & \half & s_\infty^+
\end{matrix}
\right\},
\label{eqn:P}
\end{equation}
where $s_0^\pm$ and $s_\infty^\pm$ represent the local exponents of \eqref{eqn:2ode}
 at $z=0$ and $\infty$, respectively, and are given by
\[
s_0^\pm=\pm\half\sqrt{\nu_1},\quad
s_\infty^\pm=\frac{1\pm\sqrt{4\nu_2+1}}{4}.
\]
Note that the local exponents of \eqref{eqn:2ode} at $z=1$ are 0 and $\half$.
Thus, the VE (\ref{eqn:exve1},b) is transformed
 into the Fuchsian system of the type \eqref{eqn:2ode}
 with three regular singularities,
 so that we only have to discuss the monodromy and differential Galois groups
 of \eqref{eqn:2ode} on the Riemann sphere $\Cset\cup\{\infty\}$
 (see Remark~\ref{rmk:p3}).
The following result was essentially proven in \cite{K69}.

\begin{lem}
\label{lem:ex}
Consider a Fuchsian second-order differential equation
 which has three singularities $z=z_j$, $j=1,2,3$,
 and a Riemann $P$ function
\[
P\left\{
\begin{matrix}
z_1 & z_2 & z_3\\
\rho_1^+ & \rho_2^+ & \rho_3^+ & z\\
\rho_1^- & \rho_2^- & \rho_3^-
\end{matrix}
\right\}.
\]
Then its monodromy and differential Galois groups are triangularizable
 if and only if at least one of $\rho_1+\rho_2+\rho_3$, $-\rho_1+\rho_2+\rho_3$,
 $\rho_1-\rho_2+\rho_3$ and $\rho_1+\rho_2-\rho_3$ is an odd integer,
 where $\rho_j=\rho_j^+ -\rho_j^-$, $j=1,2,3$, denote the exponent differences.
\end{lem}

\begin{figure}
\begin{center}
\includegraphics[scale=0.7]{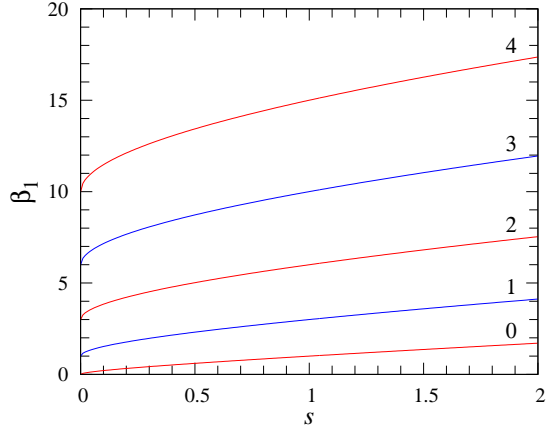}
\caption{Condition~\eqref{eqn:beta} for $\ell=$0-4.
The value of $\ell$ is labeled.
Saddle-node bifurcations can occur on the red curves
 but not on the blue curves
 while pitchfork bifurcations can occur on all curves.}
\label{fig:beta}
\end{center}
\end{figure}

Using Lemma~\ref{lem:ex}, we see that
 the monodromy and differential Galois groups for \eqref{eqn:2ode} are triangularizable
 if and only if at least one of
\begin{equation}
(s_0^+ -s_0^-)\pm(s_\infty^+ -s_\infty^-)\pm\half
 =\frac{2\sqrt{\nu_1}\pm \sqrt{4\nu_2+1}\pm 1}{2}
\label{eqn:le}
\end{equation}
is an odd integer.
Obviously, this condition is satisfied by \eqref{eqn:exve1},
 and also by \eqref{eqn:exve2} if and only if for some $\ell\in\Zset$
\[
\sqrt{8\beta_1+1}\pm 2\sqrt{s}\pm 1=2(2\ell+1),
\]
\ie,
\begin{equation}
\beta_1=\frac{(2\sqrt{s}+2\ell+1)^2-1}{8},\quad
\ell\in\Zset.
\label{eqn:beta}
\end{equation}
See Fig.~\ref{fig:beta} for how the values of $\beta_1$
 satisfying condition~\eqref{eqn:beta} with $\ell=$0-4 change when $s$ is varied.
Noting that Eq.~\eqref{eqn:sys} is $\Zset_2$-equivalent for $\beta_3=0$
 and applying Theorem~\ref{thm:main}, we prove the following theorem.

\begin{thm}
\label{thm:ex1}
\begin{enumerate}[(i)]
\item
Choose $\beta_3$ as a control parameter and fix the other parameters.
Then a saddle-node bifurcation occurs at $\beta_3=0$ in \eqref{eqn:ex}
 only if condition~\eqref{eqn:beta} holds.
\item
Choose $\beta_1$ as a control parameter and fix the other parameters,
 especially $\beta_3=\beta_4=0$.
Then a pitchfork bifurcation occurs in \eqref{eqn:ex}
 only if condition~\eqref{eqn:beta} holds.
\end{enumerate}
\end{thm}

Now we compute the second independent bounded solution of the VE (\ref{eqn:exve1},b)
 for $\beta_3=\beta_4=0$,
 based on the above result.
Consider \eqref{eqn:exve2}
 and suppose that condition~\eqref{eqn:beta} holds.
Then we have
\[
s_0^\pm=\pm\half\sqrt{s},\quad
s_\infty^\pm=\fourth\left(1\pm\left|2\sqrt{s}+2\ell+1\right|\right)
\]
in Eq.~\eqref{eqn:le}.
We set
\[
\zeta=z^{s_0^-}\eta,\quad
z^{s_0^+}(z-1)^{1/2}\eta,\quad
z^{s_0^+}\eta\quad\mbox{and}\quad
z^{s_0^-}(z-1)^{1/2}\eta,
\]
so that the Riemann $P$ function \eqref{eqn:P} becomes
\begin{align*}
&
z^{s_0^-}
P\left\{
\begin{matrix}
0 & 1 & \infty\\
s_0^- & 0 & s_\infty^- & z\\
s_0^+ & \half & s_\infty^+
\end{matrix}
\right\}
=P\left\{
\begin{matrix}
0 & 1 & \infty\\
0 & 0 & s_\infty^- +s_0^- & z\\
s_0^+ -s_0^- & \half & s_\infty^+ +s_0^-
\end{matrix}
\right\},\\
&
z^{s_0^+}(z-1)^{1/2}
P\left\{
\begin{matrix}
0 & 1 & \infty\\
s_0^- & 0 & s_\infty^- & z\\
s_0^+ & \half & s_\infty^+
\end{matrix}
\right\}
=P\left\{
\begin{matrix}
0 & 1 & \infty\\
0 & -\half & s_\infty^- +s_0^+ +\half& z\\
s_0^- -s_0^+ & 0 & s_\infty^+ +s_0^+ +\half
\end{matrix}
\right\},\\
&
z^{s_0^+}
P\left\{
\begin{matrix}
0 & 1 & \infty\\
s_0^- & 0 & s_\infty^- & z\\
s_0^+ & \half & s_\infty^+
\end{matrix}
\right\}
=P\left\{
\begin{matrix}
0 & 1 & \infty\\
0 & 0 & s_\infty^- +s_0^+ & z\\
s_0^- -s_0^+ & \half & s_\infty^+ +s_0^+
\end{matrix}
\right\}
\end{align*}
and
\[
z^{s_0^-}(z-1)^{1/2}
P\left\{
\begin{matrix}
0 & 1 & \infty\\
s_0^- & 0 & s_\infty^- & z\\
s_0^+ & \half & s_\infty^+
\end{matrix}
\right\}
=P\left\{
\begin{matrix}
0 & 1 & \infty\\
0 & -\half & s_\infty^- +s_0^+ +\half& z\\
s_0^+ -s_0^- & 0 & s_\infty^+ +s_0^+ +\half
\end{matrix}
\right\},
\]
respectively.
Hence, Eq.~\eqref{eqn:2ode} is transformed to the hypergeometric equation
\[
z(1-z)\frac{\d^2\zeta}{\d z^2}+(c-(a+b+1)z)\frac{\d\zeta}{\d z}-ab\,\zeta=0,
\]
where $c=\pm(s_0^+ -s_0^-)+1=1$ and
\begin{align*}
(a,b)=&\left(\half(\ell+1),-\sqrt{s}-\half\ell\right),\quad
\left(-\half(\ell-1),\sqrt{s}+\half\ell+1\right),\\
&
\left(-\half\ell,\sqrt{s}+\half(\ell+1)\right)\quad\mbox{and}\quad
\left(\half\ell+1,-\sqrt{s}-\half(\ell-1)\right),
\end{align*}
respectively.
Using a well-known result on the hypergeometric equation (\eg, \cite{WW27}),
 we obtain a bounded solution of \eqref{eqn:exve2} as
\begin{equation}
\xi_2(t)=z^{\sqrt{s}/2}(1-z)^{1/2}\,F\!\left(-k+1,\sqrt{s}+k+\half,1;z\right)
\label{eqn:xi2a}
\end{equation}
for $\ell=2k-1$ and
\begin{equation}
\xi_2(t)=z^{\sqrt{s}/2}\,F\!\left(-k+1,\sqrt{s}+k-\half,1;z\right)
\label{eqn:xi2b}
\end{equation}
for $\ell=2(k-1)$,
 where $k\in\Nset$, $z=\sech^2 t$
 and $F(a,b,c;z)$ is the hypergeometric function,
\[
F(a,b,c;z)=\sum_{j=0}^\infty
 \frac{a(a+1)\cdots(a+j-1)b(b+1)\cdot(b+j-1)}{j!\,c(c+1)\cdots(c+j-1)}z^j,
\]
which becomes a finite series when $a$ is a nonpositive integer.  
Note that Eq.~\eqref{eqn:2ode} also allows a solution of finite series as
\[
\xi_2(t)=z^{-\sqrt{s}/2}\,F\!\left(-k+1,-\sqrt{s}+k-\half,1;z\right)
\]
for $\ell=-2k+1$ and
\[
\xi_2(t)=z^{-\sqrt{s}/2}(z-1)^{1/2}\,F\!\left(-k+1,-\sqrt{s}+k+\half,1;z\right)
\]
for $\ell=-2k$ with $k\in\Nset$ but they are unbounded
 since $z^{-1}=\cosh^2 t\rightarrow\infty$ as $t\rightarrow 0$.
Thus, for $\beta_3=\beta_4=0$,
 if $\beta_1$ satisfies \eqref{eqn:beta} with $\ell\ge 0$,
 then condition~(C) holds and the second independent bounded solution
 of the VE (\ref{eqn:exve1},b) is given
 in the form $\varphi_2(t)=(0,\xi_2(t),0,\dot{\xi}_2(t))^*$
 by \eqref{eqn:xi2a} or \eqref{eqn:xi2b}.

We next carry out the Melnikov analysis for \eqref{eqn:ex}.
Obviously, assumptions~(M1)-(M3) and (M6) hold,
 and assumptions~(M4) and (A5) for $\beta_3=\beta_4=0$.
Using \eqref{eqn:Hab}, we have
\begin{equation}
a_2=-\int_{-\infty}^{\infty}\xi_2(t)x_1^\h(t)\d t,\quad
b_2=-\beta_4\int_{-\infty}^{\infty}\xi_2^3(t)\d t
\label{eqn:exab1}
\end{equation}
if we take $\mu=\beta_3$ as a control parameter,
 where $\xi_2(t)$ is given by \eqref{eqn:xi2a} or \eqref{eqn:xi2b}.
Since $x_1^\h(t)$ is an even function of $t$
 and $\xi_2(t)$ is an even or odd function
 depending on whether $\ell$ $(\ge 0)$ in \eqref{eqn:beta} is even or odd,
 we easily see that $a_2,b_2\neq 0$ only if $\ell$ is even and $\beta_4\neq 0$.
See Appendix~A for computations of $a_2,b_2$ when $\ell=0,2,4$.

On the other hand, assume that $\beta_3=\beta_4=0$
 and let $\mu=\beta_1$ be a control parameter.
We take
\begin{align*}
\varphi_1^*(t)=&(\sech\,t\tanh t,0,2\,\sech^3 t-\sech\,t,0),\\
\varphi_3^*(t)
 =&\bigl({\textstyle\frac{3}{2}}t\,\sech\,t\tanh t+\half\sinh t\tanh t-\sech\,t,0,\\
&\quad
 3t\,\sech^3 t+3\,\sech\,t\tanh t-{\textstyle\frac{3}{2}}t\,\sech\,t
 +\half\sinh t\bigr)\\
\psi_1^*(t)=&(3t\,\sech^3 t+3\,\sech\,t\tanh t-{\textstyle\frac{3}{2}}t\,\sech\,t
 +\half\sinh t,0,\\
& \quad
 -{\textstyle\frac{3}{2}}t\,\sech\,t\tanh t-\half\sinh t\tanh t+\sech\,t,0),\\
\psi_3^*(t)=&(-2\,\sech^3 t+\sech\,t,0,\sech\,t\tanh t,0)
\end{align*}
(see \cite{G85}) so that
\[
\xi_1^\alpha(t)
 =\varphi_{11}(t)\int_0^t\psi_{13}(\tau)x_1^\h(\tau)\xi_2^2(\tau)\d\tau
 +\varphi_{31}(t)\int_0^t\psi_{33}(\tau)x_1^\h(\tau)\xi_2^2(\tau)\d\tau,
\]
where $\varphi_{ij}(t)$ and $\psi_{ij}(t)$ are the $j$-th components
 of $\varphi_i(t)$ and $\psi_i(t)$, respectively,
while $\xi^\mu(t)=0$ since $\D_\mu f(x^\h(t);0)=0$.
Hence,
\begin{equation}
\begin{split}
&
\bar{a}_2=-\int_{-\infty}^{\infty}\xi_2^2(t)\left[x_1^\h(t)\right]^2\d t,\\
&
\bar{b}_2=-2\beta_1\int_{-\infty}^{\infty}x_1^\h(t)\xi_1^\alpha(t)\xi_2^2(t)\d t
 -\beta_2\int_{-\infty}^{\infty}\xi_2^4(t)\d t.
\end{split}
\label{eqn:exab2}
\end{equation}
Note that $\xi_1^\alpha(t)$ is an even function of $t$
 since $\varphi_{11}(t),\psi_{33}(t)$ are odd
 and $\varphi_{31}(t),\psi_{13}(t)$ are even.
We easily see that $\bar{a}_2<0$
 and $\bar{b}_2\neq 0$ for almost all values of $\beta_2$
 although the integral including $\xi_1^\alpha(t)$ in \eqref{eqn:exab2}
 is difficult to be estimated analytically.  

Applying Theorems~\ref{thm:m4} and \ref{thm:m5},
 we obtain the following theorem.

\begin{thm}
\label{thm:ex2}
\begin{enumerate}[(i)]
\item
Choose $\beta_3$ as a control parameter and fix the other parameters.
Then a saddle-node bifurcation occurs at $\beta_3=0$
 in \eqref{eqn:ex} with $\beta_4\neq 0$
 if condition~\eqref{eqn:beta} holds for a nonnegative and even integer $\ell$
 and $a_2,b_2\neq 0$.
Moreover, it is supercritical or subcritical,
 depending on whether $a_2b_2<0$ or $>0$. 
\item
Choose $\beta_1$ as a control parameter and fix the other parameters,
 especially $\beta_3=\beta_4=0$.
Then a pitchfork bifurcation occurs in \eqref{eqn:ex} with $\beta_2\neq 0$
 if condition~\eqref{eqn:beta} holds for a nonnegative integer $\ell$
 and $\bar{b}_2\neq 0$.
Moreover, it is supercritical or subcritical,
 depending on whether $\bar{b}_2>0$ or $<0$. 
\end{enumerate}
\end{thm}
In particular, if $\ell=0$ or if $\ell=2,4$ and $s$ is sufficiently large
 ($s>0.16049$ or $s\ge 5.17784\times 10^{-7}$),
 then the quantity $a_2 b_2$ has the same sign as $\beta_4$ 
 as shown in Appendix~A,
 and the saddle-node bifurcations detected by Theorem~\ref{thm:ex2}(i)
 are subcritical (resp. supercritical) for $\beta_4>0$ (resp. for $\beta_4<0$). 

\begin{figure}[t]
\begin{center}
\includegraphics[scale=0.62]{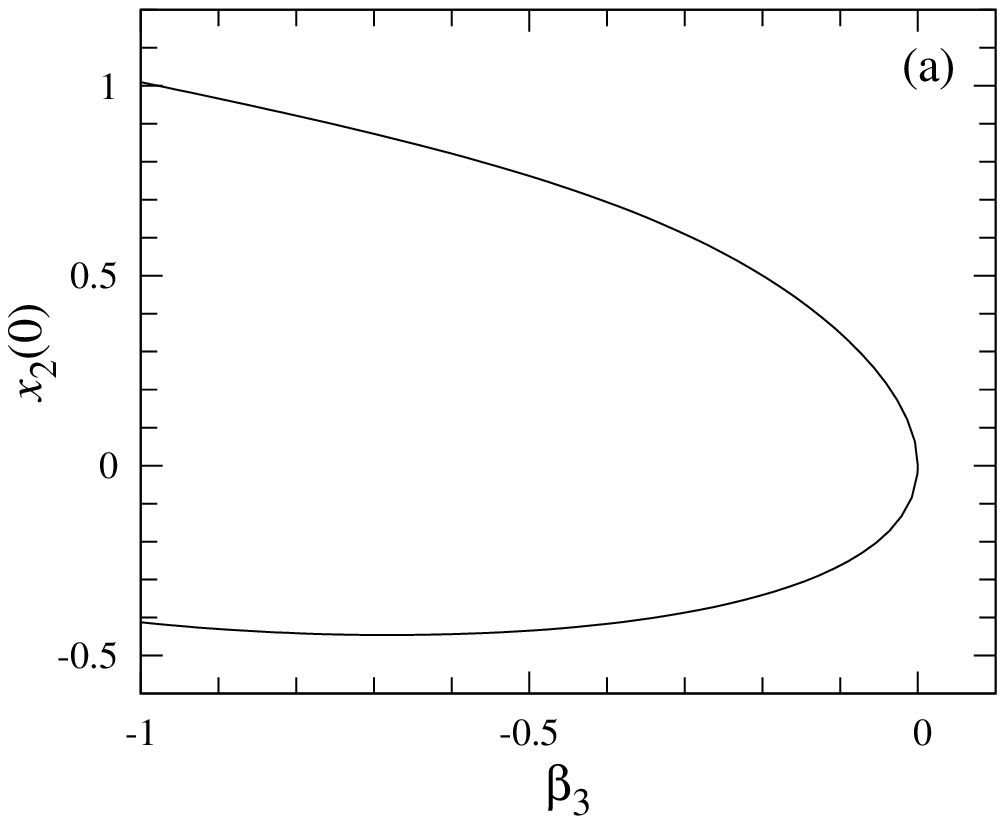}\quad
\includegraphics[scale=0.62]{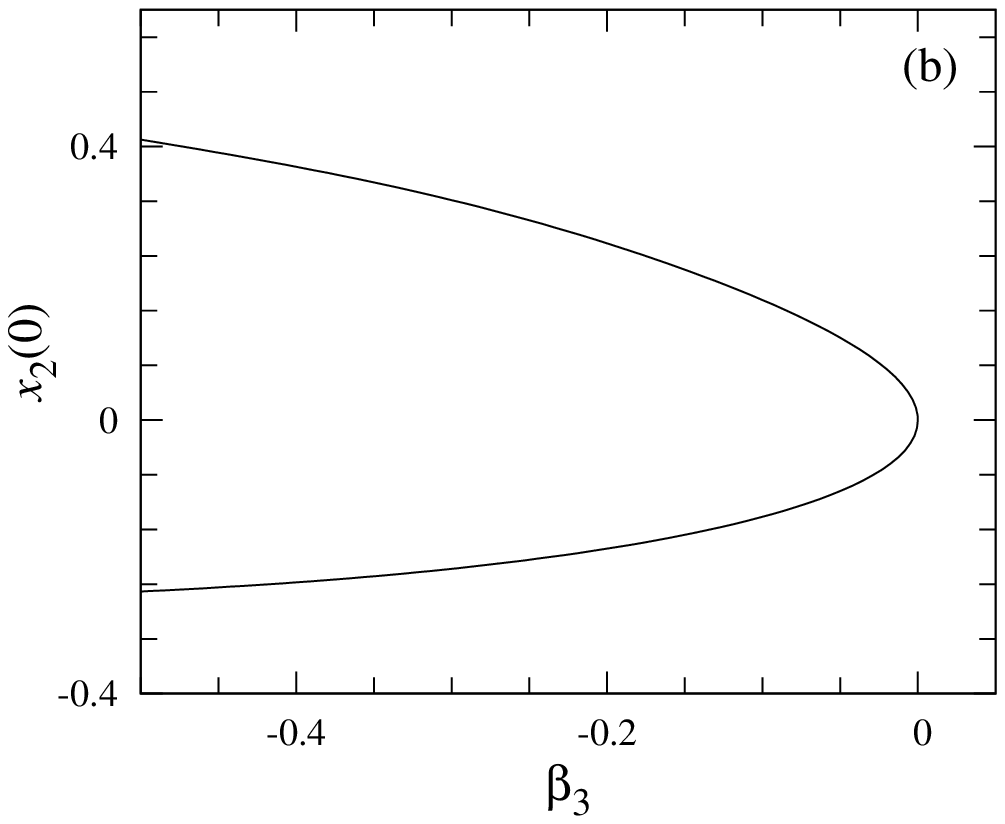}\\[2ex]
\includegraphics[scale=0.62]{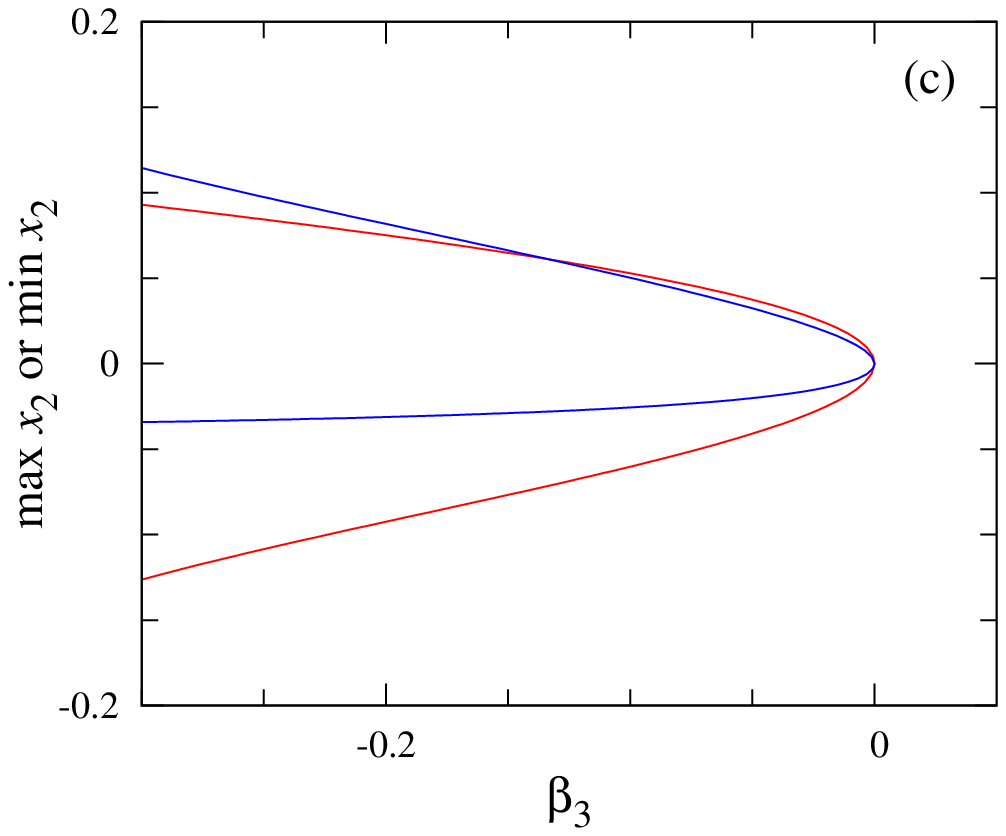}
\caption{Bifurcation diagram with $\beta_3$ a control parameter for $s=2$:
(a) $\beta_1=1.7071068$, $\beta_2=1$ and $\beta_4=2$;
(b) $\beta_1=7.5355339$, $\beta_2=1$ and $\beta_4=2$;
(c) $\beta_1=17.36396103$, $\beta_2=10$ and $\beta_4=20$.
Note that condition~\eqref{eqn:beta} approximately holds
 in plates~(a), (b) and (c) for $\ell=0,2$ and 4, respectively.
In plates~(c) two extrema varying continuously with $\beta_1$
 are plotted in red and blue.}
\label{fig:sn}
\end{center}
\end{figure}

\begin{figure}[t]
\begin{center}
\includegraphics[scale=0.52]{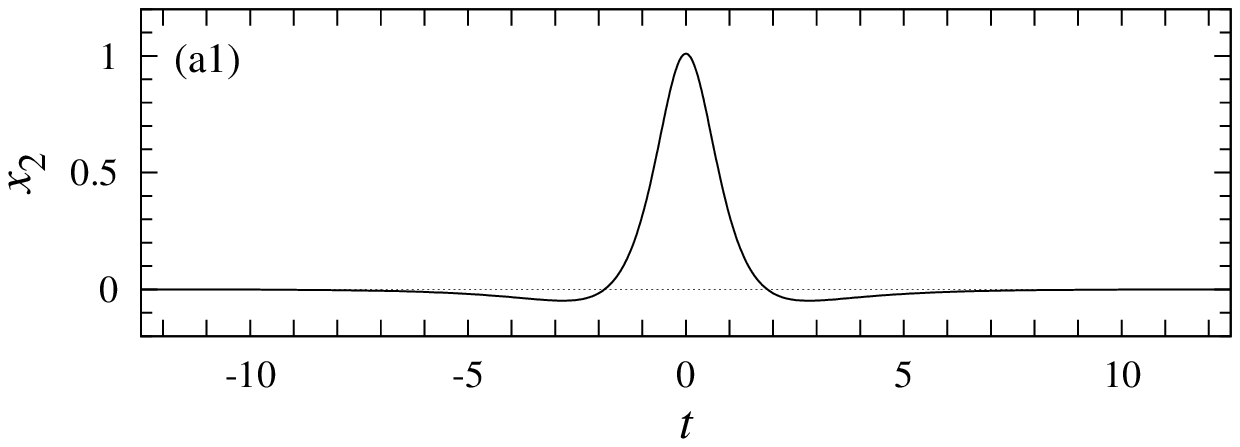}\quad
\includegraphics[scale=0.52]{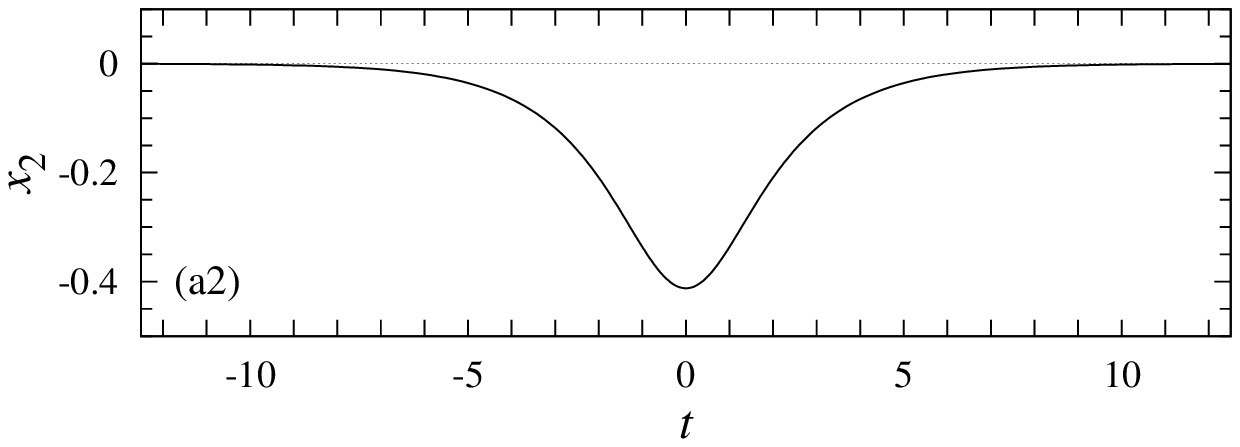}\\[2ex]
\includegraphics[scale=0.52]{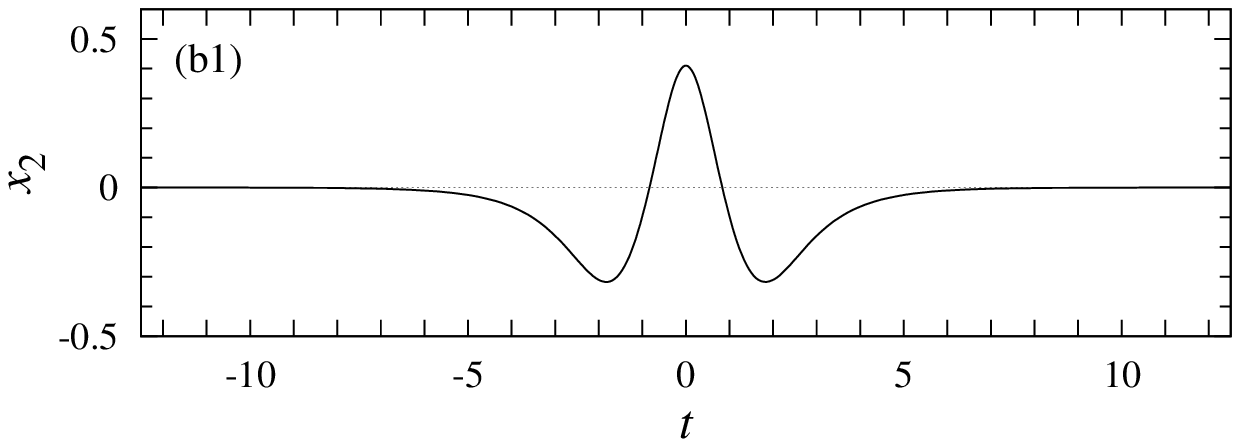}\quad
\includegraphics[scale=0.52]{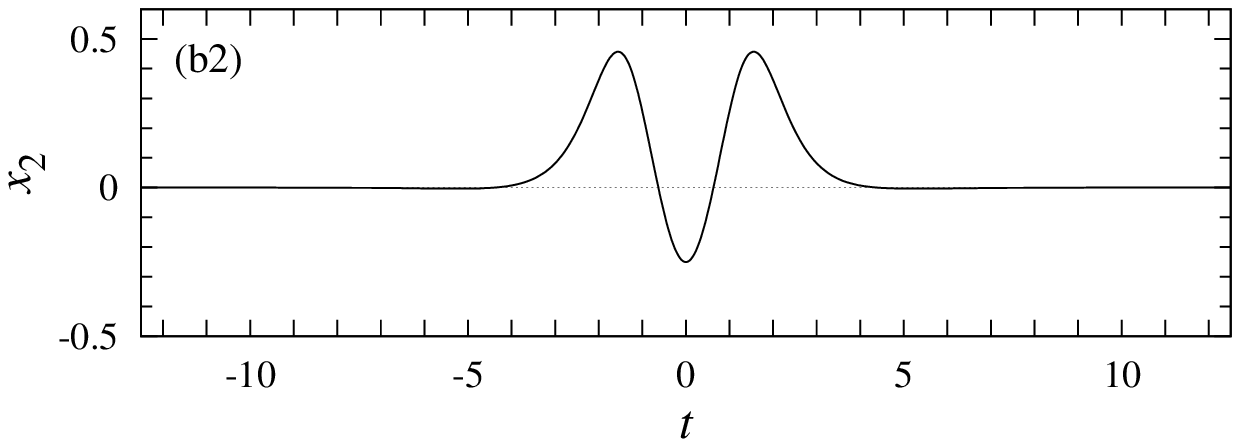}\\[2ex]
\includegraphics[scale=0.52]{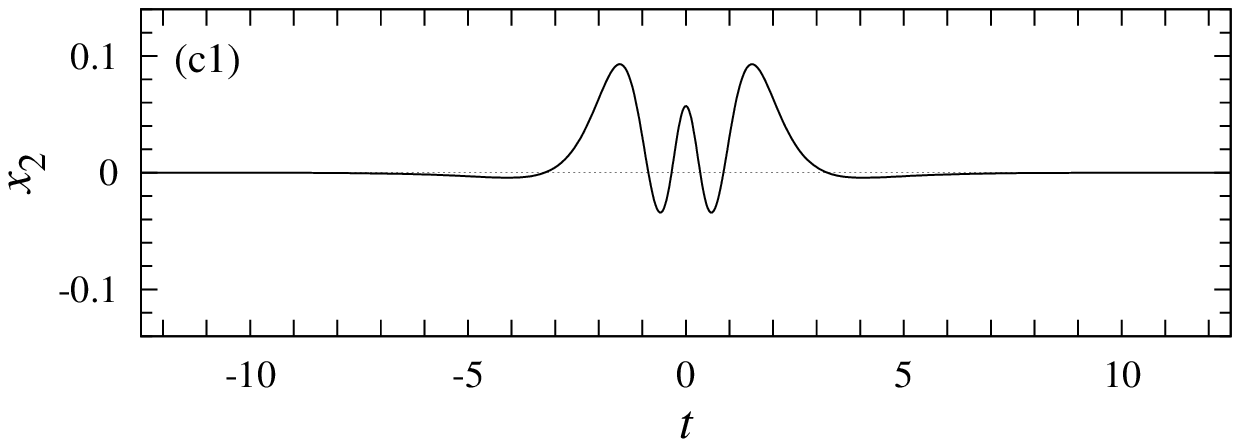}\quad
\includegraphics[scale=0.52]{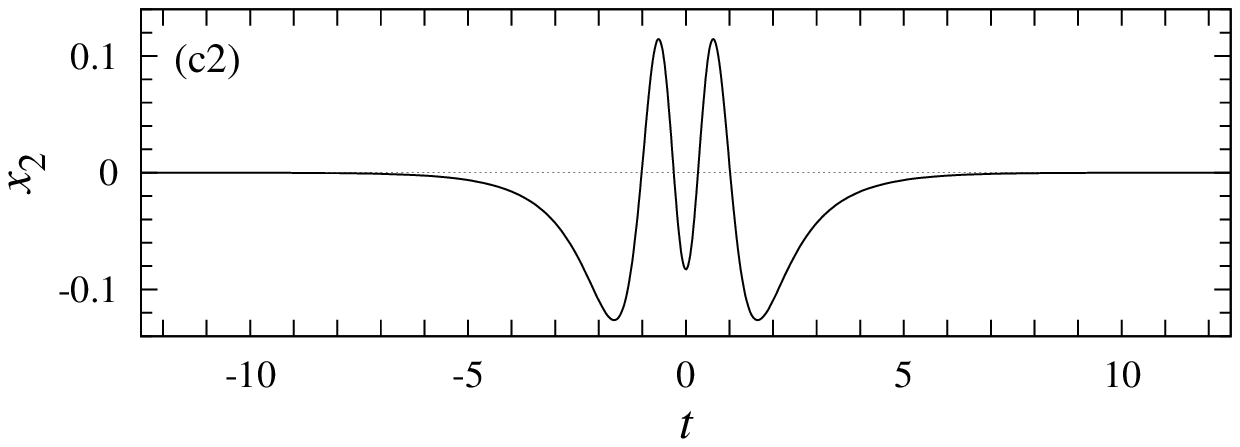}
\caption{Profiles of homoclinic orbits on the branches
 in Fig.~\ref{fig:sn}:
(a1,2) $\beta_3=1$;
(b1,2) $\beta_3=0.5$;
and (c1,2), $\beta_3=0.3$.
Other parameter values in plates~(a1,2), (b1,2) and (c1,2) are,
 respectively, the same as in plates~(a), (b) and (c) of Fig.~\ref{fig:sn}.
}
\label{fig:snprf}
\end{center}
\end{figure}

We finally give numerical results for \eqref{eqn:ex}.
We used the computation tool AUTO97 with HomCont \cite{AUTO97}
 and performed continuations of homoclinic orbits with two parameters
 in a general setting for
\begin{equation}
\begin{split}
&
\dot{x}_1=x_3,\quad
\dot{x}_3=x_1-(x_1^2+\beta_1 x_2^2)x_1-\beta_3 x_2-\nu_1 x_3,\\
&
\dot{x}_2=x_4,\quad
\dot{x}_4=sx_2-(\beta_1 x_1^2+\beta_2 x_2^2)x_2-\beta_3 x_1-\beta_4 x_2^2-\nu_1 x_4
\end{split}
\label{eqn:exd}
\end{equation}
instead of a system of the type \eqref{eqn:sys1}.
Note that
 a statement similar to that of Lemma~\ref{lem:m6} still hold in \eqref{eqn:exd}
 and the homoclinic orbit persists only if $\nu_1=0$.
The homoclinic orbit \eqref{eqn:exhomo}
 was taken as the starting solution in these continuations.
Henceforth we fix the parameter $s=2$.

Figure~\ref{fig:sn} shows branches of homoclinic orbits
 when $\beta_3$ is varied as a control parameter
 and $\beta_1$ satisfies condition~\eqref{eqn:beta} for $\ell=0,2$ and 4.
In plate~(c) of Fig.~\ref{fig:sn},
 the maximum and minimum of $x_2(t)$ are plotted
 since $x_2(t)$ has no maximum and minimum at $t=0$.
From Fig.~\ref{fig:sn} we see that saddle-node bifurcations occur at $\beta_3=0$,
 as predicted in Theorem~\ref{thm:ex2}.
Moreover, these bifurcations are subcritical,
 as predicted by Theorem~\ref{thm:ex2}(i)
 with the computations of $a_2$ and $b_2$ in Appendix~A.
Note that $x_2(t)=0$ at the bifurcation point.
We also remark that no saddle-node bifurcation was observed at $\beta_3=0$
 when $\beta_1$ satisfies condition~\eqref{eqn:beta} with $\ell=1,3$,
 and that secondary saddle-node bifurcations occur very near $\beta_3=0$
 and the branch shape becomes very different
 when $\ell$ is higher and $\beta_2,\beta_4$ are smaller
 (this is the reason why large values of $\beta_2,\beta_4$
 are taken for in plate~(c) of Fig.~\ref{fig:sn}).
Profiles of homoclinic orbits on the branches in Fig.~\ref{fig:sn}
 are given in Fig.~\ref{fig:snprf}.
Note that these homoclinic orbits are symmetric.
Asymmetric homoclinic orbits were also born from homoclinic orbits on these branches 
 at pitchfork bifurcations as well as at the secondary saddle-node bifurcations
 although no branches of asymmetric orbits are drawn in Fig.~\ref{fig:sn}.

\begin{figure}[t]
\begin{center}
\includegraphics[scale=0.62]{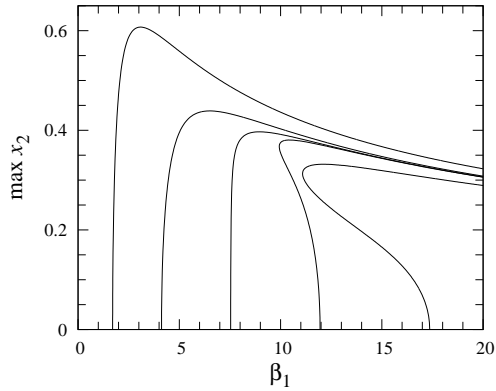}
\caption{Bifurcation diagram with $\beta_1$ a control parameter
 for $s=2$, $\beta_2=1$ and $\beta_3=\beta_4=0$.}
\label{fig:pf}
\end{center}
\end{figure}

\begin{figure}[t]
\begin{center}
\includegraphics[scale=0.58]{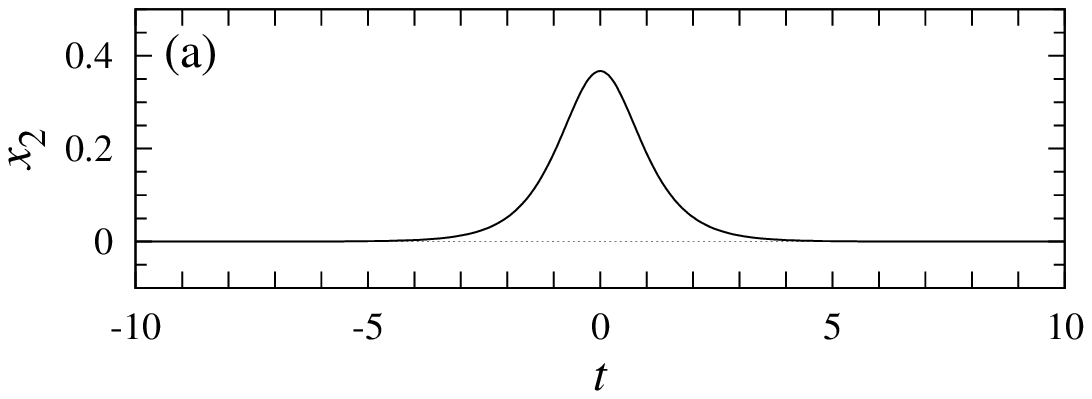}\quad
\includegraphics[scale=0.58]{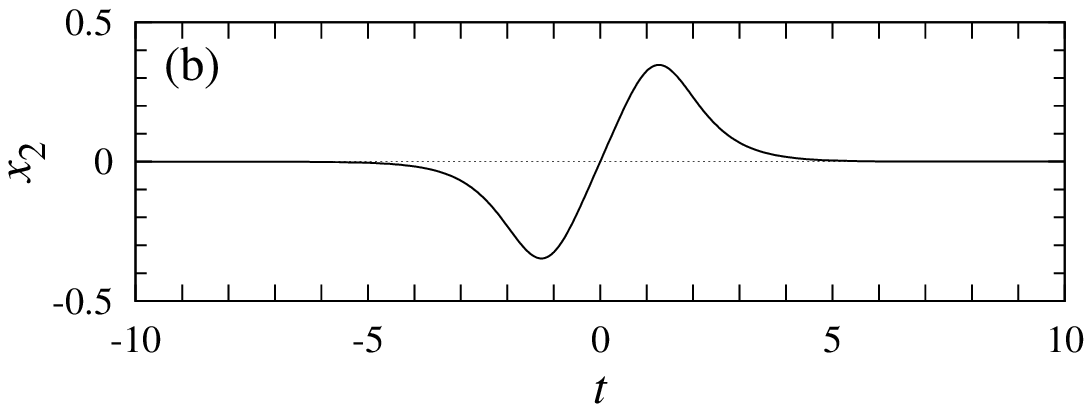}\\[2ex]
\includegraphics[scale=0.58]{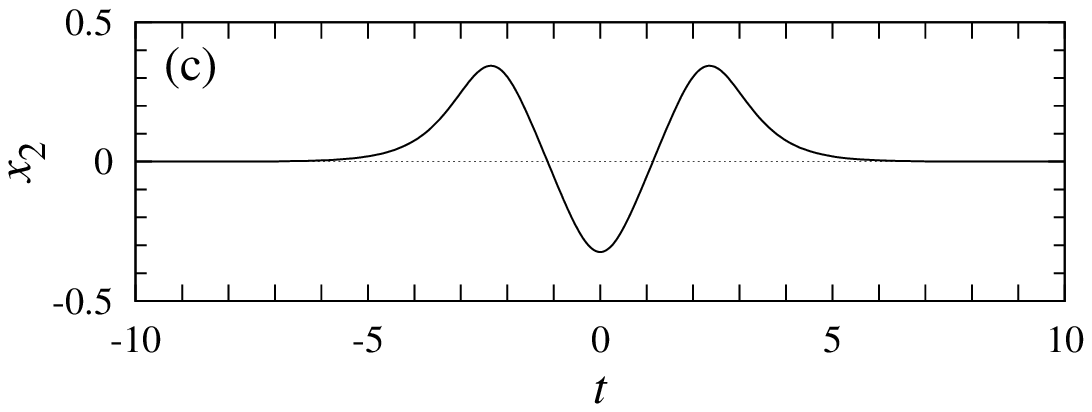}\quad
\includegraphics[scale=0.58]{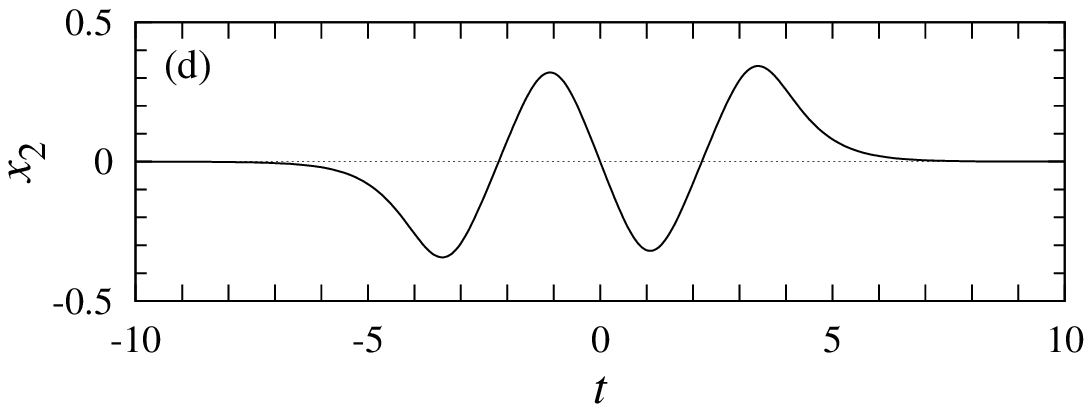}\\[2ex]
\includegraphics[scale=0.58]{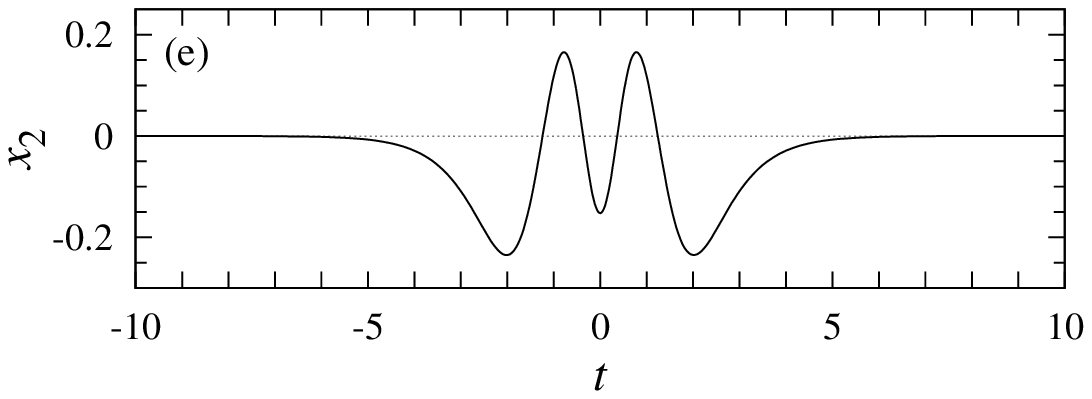}
\caption{Profiles of homoclinic orbits on the branches of $\ell=$0-4
 in Fig.~\ref{fig:pf} for $\beta_1=15$:
(a) $\ell=0$;
(b) $\ell=1$;
(c) $\ell=2$;
(d) $\ell=3$
and (e) $\ell=4$ (the lower branch).
Other parameter values are the same as in Fig.~\ref{fig:pf}.
}
\label{fig:pfprf}
\end{center}
\end{figure}

Figure~\ref{fig:pf} shows branches of homoclinic orbits
 when $\beta_1$ is varied as a control parameter
 for $\beta_2=1$ and $\beta_3=\beta_4=0$.
Note that there exist
 a branch of $x_2(=x_4)=0$ for all values of $\beta_1$, and a branch
 which is symmetric about $\max x_2=0$ to each one in Fig.~\ref{fig:pf}.
From Fig.~\ref{fig:pf} we see that pitchfork bifurcations occur
 when $\beta_1$ satisfies condition~\eqref{eqn:beta} for $\ell=$0-4,
 as predicted in Theorem~\ref{thm:ex2}.
The first three bifurcations are supercritical,
 and the rest two ones are subcritical.
Profiles of homoclinic orbits on the branches in Fig.~\ref{fig:pf}
 are given in Fig.~\ref{fig:pfprf}.
Note that homoclinic orbits on all branches in Fig.~\ref{fig:pf}
 are symmetric, like the profiles plotted in Fig.~\ref{fig:pfprf}.

\section*{Acknowledgments}
This work was partially supported
 by the Japan Society for the Promotion of Science (JSPS)
 through Grant-in-Aid for JSPS Fellows No.~21$\cdot$09222.
DBS acknowledges support
 from MICINN-FEDER grant MTM2009-06973
 and CUR-DIUE grant 2009SGR859,
 and thanks his colleagues at Instituto de Matem\'aticas y sus Aplicaciones, 
 Universidad Sergio Arboleda, for helpful discussion and encouragement.
KY appreciates support from the JSPS
 through Grant-in-Aid for Scientific Research (C) Nos.~21540124 and 22540180.

\appendix

\section{Computations of $a_2,b_2$ given by \eqref{eqn:exab1} for $\ell=0,2,4$}

We first note that
\begin{equation}
\int_{-\infty}^\infty\sech^a t\,\d t
 =\frac{2^{a-1}\Gamma^2(\half a)}{\Gamma(a)},
\label{eqn:formula}
\end{equation}
where $\Gamma(z)$ is the Gamma function,
\[
\Gamma(z)=\int_0^\infty t^{z-1}\e^{-t}\d t.
\]
Using the formula~\eqref{eqn:formula},
 we can compute $a_2,b_2$ as follows.

\subsection{Case of $\ell=0$}
We set $k=1$ in \eqref{eqn:xi2b} to have
\[
\xi_2(t)=\sech^{\sqrt{s}}t,
\]
so that
\begin{align*}
&
a_2=-\sqrt{2}\int_{-\infty}^\infty\sech^{\sqrt{s}+1}t\,\d t
 =-\frac{2^{\sqrt{s}+1/2}\Gamma^2(\half\sqrt{s}+\half)}{\Gamma(\sqrt{s}+1)},\\
&
b_2=-\beta_4\int_{-\infty}^\infty\sech^{3\sqrt{s}}t\,\d t
 =-\frac{2^{3\sqrt{s}-1}\Gamma^2(\frac{3}{2}\sqrt{s})}{\Gamma(3\sqrt{s})}\beta_4.
\end{align*}
Hence, the quantity $a_2 b_2$ has the same sign as $\beta_4$
 since $\Gamma(z)>0$ for $z>0$. 

\subsection{Case of $\ell=2$}
We set $k=2$ in \eqref{eqn:xi2b} to have
\[
\xi_2(t)=\sech^{\sqrt{s}}t
 \left(1-\left(\sqrt{s}+{\textstyle\frac{3}{2}}\right)\sech^2 t\right),
\]
so that
\begin{align*}
a_2=&-\sqrt{2}\int_{-\infty}^\infty\left(\sech^{\sqrt{s}+1}t
 -\left(\sqrt{s}+{\textstyle\frac{3}{2}}\right)\sech^{\sqrt{s}+3}t\right)\d t\\
=&-\frac{2^{\sqrt{s}+1/2}\Gamma^2(\half\sqrt{s}+\half)}{\Gamma(\sqrt{s}+1)}
 +\left(\sqrt{s}+{\textstyle\frac{3}{2}}\right)
 \frac{2^{\sqrt{s}+5/2}\Gamma^2(\half\sqrt{s}+\frac{3}{2})}{\Gamma(\sqrt{s}+3)}\\
=&\frac{2^{\sqrt{s}-1/2}(2s+3\sqrt{s}-1)\Gamma^2(\half\sqrt{s}+\half)}
 {(\sqrt{s}+2)\Gamma(\sqrt{s}+1)},\\
b_2=&-\beta_4\int_{-\infty}^\infty\Bigl(\sech^{3\sqrt{s}}t
 -3\left(\sqrt{s}+{\textstyle\frac{3}{2}}\right)\sech^{3\sqrt{s}+2}t\\
& \qquad
 +3\left(\sqrt{s}+{\textstyle\frac{3}{2}}\right)^2\sech^{3\sqrt{s}+4}t
 -\left(\sqrt{s}+{\textstyle\frac{3}{2}}\right)^3\sech^{3\sqrt{s}+6}t\Bigr)\d t\\
=&-\beta_4\biggl(
 \frac{2^{3\sqrt{s}-1}\Gamma^2(\frac{3}{2}\sqrt{s})}{\Gamma(3\sqrt{s})}
 -3\left(\sqrt{s}+{\textstyle\frac{3}{2}}\right)
 \frac{2^{3\sqrt{s}+1}\Gamma^2(\frac{3}{2}\sqrt{s}+1)}{\Gamma(3\sqrt{s}+2)}\\
& +3\left(\sqrt{s}+{\textstyle\frac{3}{2}}\right)^2
 \frac{2^{3\sqrt{s}+3}\Gamma^2(\frac{3}{2}\sqrt{s}+2)}{\Gamma(3\sqrt{s}+4)}
 -\left(\sqrt{s}+{\textstyle\frac{3}{2}}\right)^3
 \frac{2^{3\sqrt{s}+5}\Gamma^2(\frac{3}{2}\sqrt{s}+3)}{\Gamma(3\sqrt{s}+6)}
\biggr)\\
=& \frac{2^{3\sqrt{s}-4}(72s^3+252s^{5/2}+262s^2+93s^{3/2}+72s+32\sqrt{s}-40)
 \Gamma^2(\frac{3}{2}\sqrt{s})}
 {(3\sqrt{s}+1)(\sqrt{s}+1)(3\sqrt{s}+5)\Gamma(3\sqrt{s})}\beta_4,
\end{align*}
where we used the relation $\Gamma(z+1)=z\Gamma(z)$.
We see that the quantity $a_2 b_2$ has the same sign as $\beta_4$
 for $s\ge 0.16049$.

\subsection{Case of $\ell=4$}
We set $k=3$ in \eqref{eqn:xi2b} to have
\[
\xi_2(t)=\sech^{\sqrt{s}}t
 \left(1-(2\sqrt{s}+5)\sech^2 t
 +\left(\sqrt{s}+{\textstyle\frac{5}{2}}\right)
 \left(\sqrt{s}+{\textstyle\frac{7}{2}}\right)\sech^4 t\right),
\]
so that
\begin{align*}
a_2=&-\sqrt{2}\int_{-\infty}^\infty\bigl(\sech^{\sqrt{s}+1}t
 -(2\sqrt{s}+5)\sech^{\sqrt{s}+3}t\\
& \qquad
 +\left(\sqrt{s}+{\textstyle\frac{5}{2}}\right)
 \left(\sqrt{s}+{\textstyle\frac{7}{2}}\right)\sech^{\sqrt{s}+5}t\bigr)\d t\\
=&-\frac{2^{\sqrt{s}+1/2}\Gamma^2(\half\sqrt{s}+\half)}{\Gamma(\sqrt{s}+1)}
 +(2\sqrt{s}+5)\frac{2^{\sqrt{s}+5/2}\Gamma^2(\half\sqrt{s}+\frac{3}{2})}
 {\Gamma(\sqrt{s}+3)}\\
&
-\left(\sqrt{s}+{\textstyle\frac{5}{2}}\right)
 \left(\sqrt{s}+{\textstyle\frac{7}{2}}\right)
 \frac{2^{\sqrt{s}+9/2}\Gamma^2(\half\sqrt{s}+\frac{5}{2})}
 {\Gamma(\sqrt{s}+5)}\\
=&\frac{2^{\sqrt{s}-3/2}(4s^2+48s^{3/2}+199s+320\sqrt{s}+153)
 \Gamma^2(\half\sqrt{s}+\half)}
 {(\sqrt{s}+2)(\sqrt{s}+4)\Gamma(\sqrt{s}+1)},\\
b_2=&-\beta_4\int_{-\infty}^\infty\sech^{3\sqrt{s}}t
 \bigl(1-(2\sqrt{s}+5)\sech^2 t\\
&\qquad +\left(\sqrt{s}+{\textstyle\frac{5}{2}}\right)
 \left(\sqrt{s}+{\textstyle\frac{7}{2}}\right)\sech^4 t\bigr)^3\d t\\
=&-\beta_4\biggl(
 \frac{2^{3\sqrt{s}-1}\Gamma^2(\frac{3}{2}\sqrt{s})}{\Gamma(3\sqrt{s})}
 -3(2\sqrt{s}+5)\frac{2^{3\sqrt{s}+1}
 \Gamma^2(\frac{3}{2}\sqrt{s}+1)}{\Gamma(3\sqrt{s}+2)}\\
& +\left(\sqrt{s}+{\textstyle\frac{5}{2}}\right)
 \left(5\sqrt{s}+{\textstyle\frac{27}{2}}\right)
\frac{2^{3\sqrt{s}+3}
 \Gamma^2(\frac{3}{2}\sqrt{s}+2)}{\Gamma(3\sqrt{s}+4)}\\
& -2\left(\sqrt{s}+{\textstyle\frac{5}{2}}\right)^2\left(10\sqrt{s}+31\right)
\frac{2^{3\sqrt{s}+5}
 \Gamma^2(\frac{3}{2}\sqrt{s}+3)}{\Gamma(3\sqrt{s}+6)}\\
& +3\left(\sqrt{s}+{\textstyle\frac{5}{2}}\right)^2
 \left(\sqrt{s}+{\textstyle\frac{7}{2}}\right)
 \left(\sqrt{s}+{\textstyle\frac{27}{2}}\right)
\frac{2^{3\sqrt{s}+7}
 \Gamma^2(\frac{3}{2}\sqrt{s}+4)}{\Gamma(3\sqrt{s}+8)}\\
& -6\left(\sqrt{s}+{\textstyle\frac{5}{2}}\right)^3
 \left(\sqrt{s}+{\textstyle\frac{7}{2}}\right)^3
\frac{2^{3\sqrt{s}+9}
 \Gamma^2(\frac{3}{2}\sqrt{s}+5)}{\Gamma(3\sqrt{s}+10)}\\
& +\left(\sqrt{s}+{\textstyle\frac{5}{2}}\right)^3
 \left(\sqrt{s}+{\textstyle\frac{7}{2}}\right)^2
\frac{2^{3\sqrt{s}+11}
 \Gamma^2(\frac{3}{2}\sqrt{s}+6)}{\Gamma(3\sqrt{s}+12)}
\biggr)\\
=&\frac{2^{3\sqrt{s}-7}g(\sqrt{s})\Gamma^2(\frac{3}{2}\sqrt{s})}
 {(3\sqrt{s}+1)(\sqrt{s}+1)(3\sqrt{s}+5)(3\sqrt{s}+7)(\sqrt{s}+3)(3\sqrt{s}+11)
\Gamma(3\sqrt{s})}\beta_4,
\end{align*}
where
\begin{align*}
g(s)=&
 5184s^{12}+176256s^{11}+2519568s^{10}+20488032s^9+106620652 s^8\\
&
 +375344312s^7+915087795s^6+1546383098s^5+1772860056s^4\\
&
 +1308687720s^3+556461984s^2+102326688s-73920.
\end{align*}
We also see that the quantity $a_2 b_2$ has the same sign as $\beta_4$
 for $s\ge 5.17784\times 10^{-7}$.


\noindent{\sc Kazuyuki~Yagasaki \\
Mathematics Division, \\
Department of Information Engineering, \\
Niigata University,\\
Niigata 950-2181, Japan} \\
{\tt yagasaki@ie.niigata-u.ac.jp}

\medskip

\noindent{\sc David~Bl\'azquez-Sanz \\
Mathematics Division, \\
Department of Information Engineering, \\
Niigata University,\\
Niigata 950-2181, Japan} \\
\smallskip
Permanent adress: \\
{\sc Escuela de Matem\'aticas,\\
Sergio Arboleda University,\\
Bogot\'a, Colombia \\
}
{\tt david.blazquez-sanz@usa.edu.co}
\end{document}